\documentclass[12pt]{amsart}

 \usepackage{amsfonts,graphics,amsmath,amsthm,amsfonts,amscd,amssymb,amsmath,latexsym,multicol,
 mathrsfs}
\usepackage{epsfig,url}
\usepackage{flafter}
\usepackage{fancyhdr}
\usepackage{hyperref}
\hypersetup{colorlinks=true, linkcolor=black}

\makeatletter

\def\jobis#1{FF\fi
  \def\predicate{#1}%
  \edef\predicate{\expandafter\strip@prefix\meaning\predicate}%
  \edef\job{\jobname}%
  \ifx\job\predicate
}

\makeatother

\if\jobis{proposal}%
\else
\fi

 \usepackage[matrix, arrow]{xy}

\DeclareMathOperator{\lcm}{lcm}

\DeclareMathOperator{\Supp}{Supp}

\DeclareMathOperator{\vol}{vol}

 \numberwithin{equation}{subsection}
 \numberwithin{footnote}{subsection}

 \newtheorem{lem}[subsection]{Lemma}
 \newtheorem{prop}[subsection]{Proposition}
 \newtheorem{thm}[subsection]{Theorem}
 \newtheorem{conj}[subsection]{Conjecture}

{
    \newtheoremstyle{upright}%
        {8pt plus2pt minus4pt}%
        {8pt plus2pt minus4pt}%
        {\upshape}%
        {}%
        {\bfseries\scshape}%
        {}%
        {1em}%
        {}%
\theoremstyle{upright}

 \newtheorem{defn}[subsection]{Definition}

 \newtheorem{rem}[subsection]{Remark}

}

 \newcommand{\N}{\mathbb N}
 \newcommand{\PP}{\mathbb P}
 
 \newcommand{\Q}{\mathbb Q}
 \newcommand{\R}{\mathbb R}
 \newcommand{\Z}{\mathbb Z}
 \newcommand{\bir}{\dashrightarrow}
 \newcommand{\rddown}[1]{\left\lfloor{#1}\right\rfloor} 

\title[Effectivity of Iitaka fibrations]{\large E\MakeLowercase{ffectivity of}
I\MakeLowercase{itaka fibrations and pluricanonical systems of polarized pairs}}
\author{\large C\MakeLowercase{aucher} B\MakeLowercase{irkar and}
D\MakeLowercase{e}-Q\MakeLowercase{i} Z\MakeLowercase{hang}}
\date{\today}
\begin{document}
\maketitle

\begin{abstract}
For every smooth complex projective variety $W$ of dimension $d$ and nonnegative Kodaira dimension,
we show the existence of a universal constant $m$
depending only on $d$ and two natural invariants of the very general fibres of an Iitaka fibration of $W$
such that the pluricanonical system $|mK_W|$ defines an Iitaka fibration.
This is a consequence of a more general result on polarized adjoint divisors.
In order to prove these results we develop a generalized theory of pairs, singularities,
log canonical thresholds, adjunction, etc.
\end{abstract}

\tableofcontents


\section{Introduction}

We work over the complex number field $\mathbb{C}$. However, our results hold over any
algebraically closed field of characteristic zero. \\

{\textbf{\sffamily{Effectivity of Iitaka fibrations.}}}
Let $W$ be a smooth projective variety of Kodaira dimension $\kappa(W) \ge 0$.
Then by a well-known construction of Iitaka, there is a birational
morphism $V\to W$ from a smooth projective variety $V$, and a contraction $V\to X$ onto
a projective variety $X$ such that a (very) general fibre $F$ of $V\to X$
is smooth with Kodaira dimension zero, and $\dim X$ is equal to the Kodaira dimension
$\kappa(W)$. The map $W\bir X$ is referred to
as an \emph{Iitaka fibration} of $W$, which is unique up to birational equivalence.
For any sufficiently divisible natural number $m$, the pluricanonical
system $|mK_W|$ defines an Iitaka fibration.

When $\dim W = 2$, in 1970,
Iitaka [\ref{Ii}] proved that if $m$  is any natural number divisible by
$12$ and  $m\ge 86$, then $|mK_W|$ defines an Iitaka fibration (Fabrizio Catanese informed us that 
Iitaka proved this result for compact complex surfaces but the algebraic case goes back to Enriques.)
It has since been a question whether something similar holds in higher dimension. More precisely (cf. [\ref{HM}]):

\begin{conj}[Effective Iitaka fibration]\label{conj-eff-iitaka-fib}
Let $W$ be a smooth projective variety of dimension $d$ and Kodaira dimension $\kappa(W)\ge 0$.
Then there is a natural number $m_d$ depending only on $d$ such that the pluricanonical system
$|mK_W|$ defines an Iitaka fibration for any natural number $m$ divisible by $m_d$.
\end{conj}

In this paper, we show a version of the conjecture as formulated in
[\ref{VZ}, Question 0.1] holds, that is, by assuming that some invariants of the very 
general fibres of the Iitaka fibration are bounded. Without these extra assumptions 
the above conjecture seems out of reach at the moment because most likely one needs the abundance 
conjecture to deal with the very general fibres. For example, when $\kappa(W)=0$, the conjecture 
is equivalent to the effective nonvanishing $h^0(W,m_dK_W)\neq 0$ which is obviously related to the 
abundance conjecture.
Note that there is also a log version of the conjecture for pairs:
see [\ref{HX}, Conjecture 1.2, Theorem 1.4] and the references therein,
where the authors confirmed this log version when the boundary divisor is big over the generic point of
the base of the log Iitaka fibration.

We recall some definitions before stating our result. Using the notation above, let 
$W$ be a smooth projective variety of Kodaira dimension $\kappa(W)\ge 0$ and $V\to X$ an Iitaka fibration 
from a resolution $V$ of $W$.  
 For a very general fibre $F$ of  $V\to X$, let
$$
b_F := \min\{u \in \N \, | \, |uK_F| \ne \emptyset\}.
$$
Let $\widetilde{F}$ be a smooth model of the $\Z/(b_F)$-cover of $F$ ramified over
the unique divisor in $|b_FK_F|$.
Then $\widetilde{F}$ still has Kodaira dimension zero,
but with $|K_{\widetilde F}| \ne \emptyset$.
Note that
$$
\dim \widetilde{F} = \dim F = \dim W - \dim X = \dim W - \kappa(W)
$$
and we denote this number by $d_F$.
We call the Betti number
$$
\beta_{\tilde F} := \dim H^{d_F}(\tilde F, \mathbb{C})
$$
the middle Betti number of $\widetilde{F}$.\\

\begin{thm}\label{ThA}
Let $W$ be a smooth projective variety of dimension $d$ and Kodaira dimension $\kappa(W)\ge 0$.
Then there is a natural number $m(d, b_F, \beta_{\widetilde{F}})$ depending only on $d$, $b_F$ and
$\beta_{\widetilde{F}}$ such that the pluricanonical system
$|mK_W|$ defines an Iitaka fibration whenever the natural number $m$ is divisible by $m(d, b_F, \beta_{\widetilde{F}})$.
\end{thm}

The theorem is an almost immediate consequence of \ref{t-bir-bnd-M} below. The proof is
given at the end of Section 8.
When $X$ is of general type, the numbers $b_F, \beta_{\widetilde{F}}$ do not play any role so
$m(d, b_F, \beta_{\widetilde{F}})$ depends only on $d$.

Here is a brief history of
partial cases of Theorem \ref{ThA}:

$\bullet$ when $\dim W=2$ [\ref{Ii}],

$\bullet$ when $\kappa(W) = 1$ [\ref{FM}],

$\bullet$  when $W$ is of general type [\ref{HM}][\ref{Ta}] (see also [\ref{Tsuji}]),

$\bullet$   when $\kappa(W) = 2$ [\ref{VZ}] (see also [\ref{TX}]),

$\bullet$ when $\dim W=3$ [\ref{Ka86}][\ref{FM}][\ref{VZ}][\ref{HM}][\ref{Ta}] (see also [\ref{CC}]),

$\bullet$  when $X$ is non-uniruled, $V\to X$ has maximal variation and
its general fibres have good minimal models [\ref{Pa}](see also [\ref{Cerbo}]),

$\bullet$   when $V\to X$ has zero variation and
its general fibres  have good minimal models [\ref{Jiang}].

Note that the above references show that Conjecture \ref{conj-eff-iitaka-fib} 
holds when $\dim W\le 3$. 

\par \vskip 1pc

{\textbf{\sffamily{Effective birationality for polarized pairs of general type.}}}
Let $W$ be a smooth projective variety of nonegative Kodaira dimension. After replacing
$W$ birationally we can assume the Iitaka fibration $W\to X$ is a morphism.
Applying the canonical bundle formula of [\ref{FM}] (which is based on [\ref{Ka98}]), 
perhaps after replacing $W$ and $X$ birationally,
there is a $\Q$-boundary $B$ and a nef $\Q$-divisor $M$ on $X$ such that
for any natural number $m$ divisible by $b_F$ we have a natural isomorphism between
$H^0(W,mK_W)$ and $H^0(X,{m(K_X+B+M)})$.  In particular, if $|{m(K_X+B+M)}|$
defines a birational map, then $|mK_W|$ defines an Iitaka fibration.
Moreover, the coefficients of $B$ belong to a DCC set and
the Cartier index of $M$ is bounded in terms of $b_F$ and $\beta_{\widetilde{F}}$.
Therefore we can derive Theorem \ref{ThA} from the next result.

\begin{thm}\label{t-bir-bnd-M}
Let $\Lambda$ be a DCC set of nonnegative real numbers,
and $d,r$ natural numbers. Then there is a natural number $m(\Lambda, d,r)$
depending only on $\Lambda, d,r$ such that if:

\begin{itemize}
\item[(i)]
$(X,B)$ is a projective lc pair of dimension $d$,
\item[(ii)]
the coefficients of $B$ are in $\Lambda$,
\item[(iii)]
$rM$ is a nef Cartier divisor, and
\item[(iv)]
$K_X+B+M$ is big,
\end{itemize}
then the linear system $|{m(K_X+B+M)}|$ defines a birational map if $m\in \N$ is divisible by $m(\Lambda, d,r)$.
\end{thm}

We call $(X,B+M)$ a \emph{polarized pair}.
When $M=0$, the theorem is  [\ref{HMX2}, Theorem 1.3].
Note that for an $\R$-divisor $D$, by $|D|$ and $H^0(X,D)$ we mean  $|\rddown{D}|$ and
$H^0(X,\rddown{D})$.\\

{\textbf{\sffamily{Generalized polarized pairs.}}}
In order to prove Theorem \ref{t-bir-bnd-M} we need to generalize the definitions of  pairs,
singularities, lc thresholds, adjunction, etc.  We develop
this theory, which is of independent interest, in some detail in Section 4 but for now we only
give the definition of generalized polarized pairs.

\begin{defn}\label{d-g-pol-pair}
A \emph{generalized polarized pair} consists of a normal variety $X'$ equipped with projective
morphisms $X \overset{f}\to X'\to Z$ where $f$ is birational and $X$ is normal, an $\R$-boundary $B'$,
 and an $\R$-Cartier divisor
$M$ on $X$ which is nef$/Z$ such that $K_{X'}+B'+M'$ is $\R$-Cartier,
where $M' := f_*M$. We call $B'$ the \emph{boundary part} and
$M$ the \emph{nef part}.

Note that the definition is flexible with respect to $X,M$. To be more precise,
if $g\colon Y\to X$ is a projective birational morphism
from a normal variety, then there is no harm in replacing $X$ with $Y$ and replacing $M$ with $g^*M$.

For us the most interesting case is when $M=\sum \mu_j M_j$ where $\mu _j \ge 0$ and
$M_j$ are nef$/Z$ Cartier divisors. In many ways $B'+M'$ behaves like a boundary, that is,
it is as if the $M_j'$ were components of the boundary with coefficient $\mu_j$.
Although the coefficients of $B_i'$ belong to
the real interval $[0,1]$ the coefficients $\mu_j$ are only assumed to be nonnegative. Moreover, the $M_j$
are not necessarily distinct. See Section 4 for more details.
\end{defn}

When $X\to X'$ is the identity morphism, we recover the definition of
polarized pairs which was formally introduced in [\ref{BH}] but appeared earlier in the 
subadjunction formula of [\ref{Ka98}]. 
If moreover $M=0$, then $(X',B')$ is just a pair in the traditional sense.\\

{\textbf{\sffamily{ACC for generalized lc thresholds.}}}
The next result shows that the generalized lc thresholds satisfy ACC under suitable
assumptions. We suggest the reader looks at Definitions \ref{d-g-sing} and \ref{d-g-lct}
before continuing.

\begin{thm}\label{t-acc-glct}
Let $\Lambda$ be a DCC set of nonnegative real numbers and $d$ a natural number.
Then there is an ACC set $\Theta$ depending only on $\Lambda,d$ such that
if $(X',B'+M')$, $M$, $N$, and $D'$ are as in Definition \ref{d-g-lct} satisfying

\begin{itemize}
\item[(i)]
$(X',B'+M')$ is generalized lc of dimension $d$,
\item[(ii)]
$M=\sum \mu_jM_j$ where $M_j$ are nef$/Z$ Cartier divisors and $\mu_j\in \Lambda$,
\item[(iii)]
$N=\sum \nu_kN_k$ where $N_k$ are nef$/Z$ Cartier divisors and $\nu_k\in \Lambda$, and
\item[(iv)]
the coefficients of $B'$ and $D'$ belong to $\Lambda$,
\end{itemize}
then the generalized lc threshold of $D'+N'$ with respect to $(X',B'+M')$ belongs to $\Theta$.
\end{thm}

Note that the theorem is a local statement over $X'$,  so $Z$ does not play any role and we
could simply assume $X'\to Z$ is the identity map.

When $X \to X'$ is the identity map, $M = 0$, and $N = 0$, the theorem is 
the usual ACC for lc thresholds [\ref{HMX2}, Theorem 1.1].\\

{\textbf{\sffamily{Global ACC.}}}
The proof of the previous result requires the following global ACC.
We will also use this to bound pseudo-effective thresholds (Theorem \ref{t-peff-th-B+M})
which is in turn used in the proof of Theorem \ref{t-bir-bnd-M}.

\begin{thm}\label{t-global-acc}
Let $\Lambda$ be a DCC set of nonnegative real numbers and $d$ a natural number.
Then there is a finite subset $\Lambda^0\subseteq \Lambda$ depending only on $\Lambda,d$ such that
if $(X',B'+M')$, $X\to X'\to Z$ and $M$ are as in Definition \ref{d-g-pol-pair} satisfying
\begin{itemize}
\item[(i)]
$(X',B'+M')$ is generalized lc of dimension $d$,
\item[(ii)]
$Z$ is a point,
\item[(iii)]
$M=\sum \mu_jM_j$ where $M_j$ are nef Cartier divisors and $\mu_j\in \Lambda$,
\item[(iv)]
$\mu_j=0$ if $M_j\equiv 0$,
\item[(v)]
the coefficients of $B'$ belong to $\Lambda$, and
\item[(vi)]
$K_{X'}+B'+M'\equiv 0$,
\end{itemize}
then the coefficients of $B'$ and the $\mu_j$ belong to $\Lambda^0$.
\end{thm}

When $X \to X'$ is the identity map and $M = 0$, the theorem is [\ref{HMX2}, Theorem 1.5].\\

{\textbf{\sffamily{About this paper.}}}
It is not hard to reduce Theorems \ref{t-bir-bnd-M} and \ref{t-acc-glct} to Theorem \ref{t-global-acc}.
So most of the difficulties we face have to do with \ref{t-global-acc}.
Since the statement of Theorems \ref{t-bir-bnd-M}, \ref{t-acc-glct}, and \ref{t-global-acc} involve
nef divisors which may not be semi-ample (or effectively semi-ample), there does not seem to be
any easy way to reduce them to the traditional versions (i.e. without nef divisors) proved
in [\ref{HMX2}] or to mimic the arguments in [\ref{HMX2}]. 
Instead we need to develop new ideas and arguments and this occupies much of this paper. 

We briefly explain the organization of the paper.
In Section 3, we prove a special case of Theorem \ref{t-bir-bnd-M} (Proposition \ref{p-bir-nM-n-large})
by closely following [\ref{HMX2}].
In Section 4, we introduce generalized singularities and generalized lc thresholds,
discuss the log minimal model program for generalized polarized pairs,
and treat generalized adjunction.
In Section 5, we give bounds, both in the local and global situations, on the numbers of
components in the boundary and nef parts of generalized polarized pairs, under appropriate assumptions. 
These bounds will be used in the proof of Proposition \ref{p-glct-to-global-acc} which serves as the key inductive 
step toward the proof of Theorem \ref{t-global-acc}. In Section 6, we reduce Theorem 
 \ref{t-acc-glct} to Theorem \ref{t-global-acc} in lower dimension by adapting a standard argument. 
In Section 7, we treat Theorem \ref{t-global-acc} inductively where we apply Proposition \ref{p-bir-nM-n-large}; 
a sketch of the main ideas is included below in this introduction. 
In Section 8, we give the proofs of our main results.  Theorems \ref{t-acc-glct} and \ref{t-global-acc} 
follow immediately from Sections 6 and 7. 
To prove Theorem \ref{t-bir-bnd-M}, we use Theorem \ref{t-global-acc}
to bound certain pseudo-effective thresholds (Theorem \ref{t-peff-th-B+M}) and use the concept of 
potential birationality [\ref{HMX2}] to reduce to the special case of Proposition \ref{p-bir-nM-n-large}.
Finally, we extend Theorem \ref{t-bir-bnd-M} to
allow more general coefficients in the nef part of the pair (see Theorem 8.2), and
deduce Theorem \ref{ThA} from Theorem \ref{t-bir-bnd-M} as in [\ref{FM}][\ref{VZ}].\\

{\textbf{\sffamily{A few words about the proof of Theorem \ref{t-global-acc}.}}}
We try to explain, briefly, some of the ideas used in the proof of \ref{t-global-acc}.
By [\ref{HMX2}, 1.5] we can assume $M\not \equiv 0$. The basic strategy is to
modify $(X',B'+M')$ so that the nef part has one less coefficient $\mu_j$ and
then repeat this to reach the case $M=0$.
Running appropriate LMMP's we can reduce the problem to the case when
$X'$ is a $\Q$-factorial klt Fano variety with Picard number one.
Moreover, some lengthy arguments show that the number of the
$\mu_j$ is bounded (Section \ref{s-bnd-comp}). If $(X',B'+M')$ is not
generalized klt, one can do induction: for example if $\rddown{B'}\neq 0$,
then we let $S'$ be the normalization of a component of $\rddown{B'}$ and
use generalized adjunction (see Definition \ref{d-q-adjunction}) to write
$$
K_{S'}+B_{S'}+M_{S'}=(K_{X'}+B'+M')|_{S'}
$$
and apply induction to the generalized lc polarized pair $({S'},B_{S'}+M_{S'})$.
So we can assume  $(X',B'+M')$ is generalized klt.

Although we cannot use the arguments of [\ref{HMX2}] to prove Theorem \ref{t-global-acc}
but there is an exception: if we take $n\in \N$ to be sufficiently large, then following
[\ref{HMX2}] closely one can show that there is $m\in\N$ depending only on $\Lambda, d$ such that
$|m(K_X+B+\sum nM_j|)|$ defines a birational map (Proposition \ref{p-bir-nM-n-large}) where 
$B$ is the sum of the birational transform of $B'$ and the reduced exceptional divisor 
of $X\to X'$.
One can then show that there is an $\R$-divisor $D$ such that
$$
0\le D\sim_\R K_X+B+\sum nM_j
$$
where the coefficients of $D$ belong to some DCC set depending only on $\Lambda, d$.
Then the pushdown $D'$ of $D$ satisfies
$$
D'\sim_\R K_{X'}+B'+\sum nM_j'\equiv \sum (n-\mu_j)M_j'\equiv \rho M_{1}'
$$
for some number $\rho$. Changing the indexes one can assume that $\rho$ belongs to some
ACC set depending only on $\Lambda, d$. Let $N=M-\mu_1M_1$. Now the idea is to take $s,t$,
with $s$ maximal, so that
$$
 K_{X'}+B'+sD'+N'+tM_1' \equiv K_{X'}+B'+M'
$$
and that $({X'},B'+sD'+N'+tM_1')$ is generalized lc. If it happens to have $t=0$, then
$s$ would belong to some DCC set and we can replace $B'$ with $B'+sD'$ and replace $M$
with $N$ which has one less summand, and repeat the process. But if $t>0$,
then $({X'},B'+sD'+N'+tM_1')$ is generalized lc but not generalized klt. We cannot simply
apply induction because the $s,t$ may not belong to a DCC set. For simplicity assume
$\rddown{B'+sD'}\neq 0$ and let $S'$ be one of its components and assume $S'$ is normal.
The idea is to keep $S'$ but to remove the other components of $D'$ and
increase $t$ instead so that we get
$$
 K_{X'}+B'+\tilde{s}S'+N'+\tilde{t}M_1' \equiv K_{X'}+B'+M'
$$
for some $\tilde{s}$ and $\tilde{t}\ge t$ where $S'$ is a component of $\rddown{B'+\tilde{s}S'}$.
Now it turns out $\tilde{t}$ belongs to some DCC set
and we can apply induction by restricting to $S'$.\\

{\textbf{\sffamily{Acknowledgements.}}}
The first author was partially supported by a grant of the Leverhulme Trust.
Part of this work was done when the first author visited National University of Singapore
in April 2014.
Part of this work was done when the first author visited National Taiwan University in August-September 2014
with the support of the
Mathematics Division (Taipei Office) of the National Center for Theoretical Sciences.
The visit was arranged by Jungkai A. Chen.
 He wishes to thank them all.
The second author was partially supported by an ARF of National University of Singapore.
The authors would like to thank the referee for the very useful corrections and suggestions 
which helped to simplify and clarify some of the proofs.

%
%
%

\section{Preliminaries}

\par \vskip 1pc

{\textbf{\sffamily{Notation and terminology.}}} All the varieties in this paper are quasi-projective
over $\mathbb{C}$ unless stated otherwise.
For definitions and basic properties of singularities of pairs such as log canonical
(lc), Kawamata log terminal (klt), divisorially log terminal (dlt), purely log terminal (plt),
and the log minimal model program
(LMMP) we refer to [\ref{KM}]. We recall some notation:

\begin{itemize}
\item
The sets of natural, integer, rational, and real numbers are respectively denoted as
$\N, \Z, \Q, \R$.

\item
Divisors on normal varieties are always Weil $\R$-divisors unless otherwise stated.

\item
Let $X \to Z$ be a projective morphism from a normal variety.
\emph{Linear equivalence}, $\Q$-\emph{linear equivalence},
$\R$-\emph{linear equivalence}, and \emph{numerical equivalence} over $Z$,
between two $\R$-divisors $D_1,D_2$
on $X$ are respectively denoted as $D_1 \sim D_2/Z$, $D_1 \sim_{\Q} D_2/Z$, $D_1 \sim_{\R} D_2/Z$, and
$D_1 \equiv D_2/ Z$. If $Z$ is a point, we usually drop the $Z$.

\item
If $\phi\colon X \bir X'$ is a birational morphism whose inverse does not contract divisors,
and $D$ is an  $\R$-divisor on $X$,
we usually write $D'$ for  $\phi_*D$. If $X'$ is replaced by $X''$ (resp. $Y$) we usually
write  $D''$ (resp. $D_Y$) for $\phi_*D$.

\item Let $X,Y$ be normal varieties projective over some base $Z$, and $\phi\colon X\bir Y$ a
birational map$/Z$ whose inverse does not contract any divisor.
Let $D$ be an $\R$-Cartier divisor on $X$ such that $D_Y$ is also $\R$-Cartier.
We say $\phi$ is \emph{$D$-negative} if there is a common resolution $g\colon W\to X$ and $h\colon W\to Y$ such that
$E:=g^*D-h^*D_Y$ is effective and exceptional$/Y$, and
$\Supp g_*E$ contains all the exceptional divisors of $\phi$.
\end{itemize}

\par \vskip 1pc

{\textbf{\sffamily{ACC and DCC sets.}}}
A sequence $\{a_i\}$ of numbers is {\it increasing} (resp. {\it strictly increasing})
if $a_i \le a_{i+1}$ (resp. $a_i < a_{i+1}$) for all $i$.
The definition of a decreasing or strictly decreasing sequence is similar.
A set $\Lambda$ of real numbers satisfies DCC (descending chain condition)
if it does not contain a strictly decreasing infinite sequence.
A set $\Omega$ of real numbers satisfies ACC (ascending chain condition)
if it does not contain a strictly increasing infinite sequence.

\begin{lem}\label{l-ACC-DCC}
Let $\Lambda$ and $\Omega$ be sets of nonnegative real numbers.
Define
$$
\Lambda+\Omega=\{a+b \mid a\in \Lambda, b\in \Omega\}
$$
and
$$
\Lambda\cdot \Omega=\{ab \mid a\in \Lambda, b\in \Omega\} .
$$
Then the following hold:

$(1)$ If $\Lambda$ and $\Omega$ are both ACC sets (resp. DCC sets), then
$\Lambda+\Omega$ and $\Lambda\cdot\Omega$ are also ACC sets (resp. DCC sets).

$(2)$ Let $\{a_i\} \subseteq \Lambda$ and $\{b_i\} \subseteq \Omega$ be sequences of numbers.
Assume that both sequences are increasing and that one of them is strictly
increasing. Then the sequences $\{a_i+b_i\}$ and $\{a_ib_i\}$ are strictly increasing.

$(3)$ A statement similar to $(2)$ holds if we replace `increasing' by `decreasing'.

$(4)$ Let $m,l\in \N$. Assume that $\Lambda$ is a DCC set and that $a\le l$
for every $a\in \Lambda$.  Then the set $\{\langle ma\rangle \, | \, a \in \Lambda \}$
also satisfies DCC, where $\langle ma\rangle:=ma-\rddown{ma}$, that is, the fractional
part of $ma$.
\end{lem}
\begin{proof}
The proof is left to the reader. \\
\end{proof}

\begin{lem}\label{l-ordering-divs}
Let $d,r$ be natural numbers.
Let $X_i$ be a sequence of normal projective varieties of dimension $d$ and Picard number one.
Assume that $D_{1,i}, \dots, D_{r,i}$ are nonzero $\R$-Cartier divisors on $X_i$. Let
$\lambda_{j,i}$ be the numbers such that $D_{j,i}\equiv\lambda_{j,i}D_{1,i}$.
Then possibly after
replacing the sequence with an infinite subsequence and rearranging the indexes,
the sequence $\lambda_{j,i}$ is a decreasing sequence for each $j$.
\end{lem}
\begin{proof}
Let $\rho_{j,k,i}$ be the numbers such that $D_{j,i}\equiv \rho_{j,k,i}D_{k,i}$.
Replacing the sequence we may assume that for each $j,k$ the sequence $\rho_{j,k,i}$ is increasing
or decreasing.
If $\rho_{j,k,i}$ is decreasing we write $j\trianglelefteq k$. This relation is
associative, that is, if $j\trianglelefteq k$ and $k\trianglelefteq l$,
then $j\trianglelefteq l$ because $\rho_{j,l,i}=\rho_{j,k,i}\rho_{k,l,i}$. So we can order the sequences of divisors
according to this relation. Changing the indexes we may assume that
$r\trianglelefteq \cdots \trianglelefteq 1$ which in particular means that the
$\lambda_{j,i}=\rho_{j,1,i}$ form a decreasing sequence for each $j$.\\
\end{proof}

{\textbf{\sffamily{Minimal models and Mori fibre spaces.}}}\label{ss-pair-mmodel}
Let $ X\to Z$ be a
projective morphism of normal varieties and $D$ an $\R$-Cartier divisor
on $X$. A normal variety $Y$ projective over $Z$ together with a birational map $\phi\colon X\bir Y/Z$
whose inverse does not contract any divisor is called a \emph{minimal model of $D$ over $Z$} if:

$(1)$ $Y$ is $\Q$-factorial,

$(2)$ $D_Y=\phi_*D$ is nef$/Z$, and

$(3)$ $\phi$ is $D$-negative.

If one can run an LMMP on $D$ over $Z$ which terminates with a $\Q$-factorial model $Y$ on which $D_Y$ is nef$/Z$,
then $Y$ is a minimal model of $D$ over $Z$.

On the other hand, we call $Y$ a \emph{Mori fibre space} of $D$ over $Z$ if $Y$ satisfies the
above conditions with condition $(2)$ replaced by:

$(2)'$ there is an extremal contraction
$Y\to T/Z$ such that $-D_Y$ is ample$/T$.

In practice, we consider minimal models
and Mori fibre space for $K_{X'}+B'+M'$ where $(X',B'+M')$ is a generalized polarized pair. \\

{\textbf{\sffamily{Some notions and results of [\ref{HMX2}].}}}
For convenience we recall some technical notions and results of [\ref{HMX2}]
which will be used in Section 3.

Let $X$ be a normal projective variety, and let $D$ be a big
$\Q$-Cartier $\Q$-divisor on $X$. We say that $D$ is \emph{potentially
birational} [\ref{HMX2}, Definition 3.5.3]
if for any pair $x$ and $y$ of general points of $X$, possibly
switching $x$ and $y$, we can find $0 \le  \Delta \sim_\Q (1 - \epsilon)D$
for some $0 < \epsilon < 1$ such that
$(X, \Delta)$ is not klt at $y$ but $(X, \Delta)$ is lc at $x$ and
$\{x\}$ is a non-klt centre.

\begin{thm}[{[\ref{HMX2}, Theorem 3.5.4]}]\label{t-HMX3.5.4}
Let $(X, B)$ be a klt pair, where X
is projective of dimension $d$, and let $H$ be an ample $\Q$-divisor. Suppose there
exist a constant $\gamma \ge 1$ and a family $V \to C$ of subvarieties of $X$ with the following
property:
if $x$ and $y$ are two general points of $X$ then, possibly switching $x$ and $y$, we
can find $c \in C$ and $0 \le \Delta_c \sim_\Q (1 -\delta)H$, for some $\delta > 0$,
such that $(X, B + \Delta_c)$
is not klt at $y$ and there is a unique non-klt place of $(X, B + \Delta_c)$ whose
centre $V_c$ contains $x$. Further assume there is
a divisor $D$ on $W$, the normalization of $V_c$, such that the linear system $|D|$ defines a
 birational map and
$\gamma H|_W -D$ is pseudo-effective.
Then $mH$ is potentially birational, where $m = 2p^2 \gamma + 1$ and $p = \dim V_c$.
\end{thm}

\begin{thm}[{[\ref{HMX2}, Theorem 4.2]}]\label{t-HMX4.2}
Let $\Lambda$ be a subset of $[0, 1]$ which contains $1$. Let $X$ be a
projective variety of dimension $d$, and let $V$ be a
subvariety, with normalization $W$. Suppose we are given an $\R$-boundary $B$ and an
$\R$-Cartier divisor $G\ge 0$, with the following properties:
\begin{itemize}
\item[(1)]
the coefficients of $B$ belong to $\Lambda$;
\item[(2)]
$(X, B)$ is klt; and
\item[(3)]
there is a unique non-klt place $\nu$ for $(X, B + G)$, with
centre $V$.
\end{itemize}
Then there is an $\R$-boundary $B_W$ on $W$ whose coefficients belong to
$$
\{a \, | \, 1-a \in LCT_{d-1}(D(\Lambda))\} \cup \{1\}
$$
such that the difference
$$
(K_X + B + G) |_W - (K_W + B_W)
$$
is pseudo-effective.

Now suppose that $V$ is the general member of a covering family of subvarieties of $X$.
Let $\psi : U \to W$ be a log resolution of $(W,B_W)$, and let $B_U$ be the sum
of the birational transform of $B_W$ and the reduced exceptional divisor of $\psi$. Then
$$
K_U + B_U \ge (K_X + B) |_U .
$$
The notation $|_W$ and $|_U$ mean pullback to $W$ and $U$ respectively.
\end{thm}

\begin{rem}\label{rem-Lambda'}
Assume that the $\Lambda$ in \ref{t-HMX4.2} satisfies DCC. Then the hyperstandard set $D(\Lambda)$
also satisfies DCC, hence the set of lc thresholds
$LCT_{d-1}(D(\Lambda))$
satisfies ACC by the ACC for usual lc thresholds [\ref{HMX2}, Theorem 1.1]. Therefore, the set
$$
\{a \, | \, 1-a \in LCT_{d-1}(D(\Lambda))\} \cup \{1\}
$$
to which the coefficients of $B_U$ belong, also satisfies DCC.\\
\end{rem}

%
%
%

\section{Effective birationality of $K + B + nM$}

In this section, following [\ref{HMX2}] closely, we prove a special case of Theorem \ref{t-bir-bnd-M} (see \ref{p-bir-nM-n-large})
which will be used in Sections \ref{s-global-acc} and 8 in proving 
Theorems \ref{t-global-acc} and \ref{t-bir-bnd-M}. 
This special case concerns effective birationality 
for big divisors of the form $K_X+B+nM$ where $(X,B)$ is projective lc, $rM$ is nef and Cartier, 
and $n/r$ is large enough. Running an LMMP on $K_X+B+nM$ preserves the nef and Cartier properties of 
$rM$ by boundedness of length of extremal rays [\ref{Ka91}] which allows one to apply the methods of 
[\ref{HMX2}]. In contrast if one runs an LMMP on $K_X+B+M$, the nef and Cartier properties of 
$rM$ may be lost, hence one needs to consider generalized polarized pairs which will be discussed in 
later sections. 

First we prove a few lemmas.

\begin{lem}\label{l-addNef}
Let $X$ be a normal projective variety, $D$ a big
$\Q$-Cartier $\Q$-divisor, and $G$ a nef $\Q$-Cartier $\Q$-divisor on $X$.
If $D$ is potentially birational,
then $D+G$ is also potentially birational.
In particular, $|K_X + \lceil D + G \rceil|$ defines a birational map.
\end{lem}
\begin{proof}
Write $D \sim_\Q A + B$ with $B$ effective and $A$ ample.
By definition, for any pair $x,y\in X$ of general points, possibly after switching $x,y$,
there exist $\epsilon\in(0,1)$ and a $\Q$-divisor $0\le \Delta \sim_{\Q} (1- \epsilon) D$
such that $(X,\Delta)$ is not klt at $y$ but it is lc at $x$ and $\{x\}$ is a non-klt centre.
Now if $\epsilon' \in(0,\epsilon)$ is rational, then we can find
$$
0\le \Delta'\sim_\Q \Delta+(\epsilon-\epsilon')B+(\epsilon-\epsilon')A+(1-\epsilon')G\sim_\Q (1-\epsilon')(D+G)
$$
so that $(X,\Delta')$ satisfies the same above properties as $(X,\Delta)$ at $x,y$.
So $D+G$ is potentially birational. To get the last claim, just apply [\ref{HMX}, Lemma 2.3.4 (1)].\\
\end{proof}

\begin{lem}\label{l-peffth-of-B-in-K+B+nM}
Let $\Lambda$ be a DCC set of nonnegative real numbers, and $d,r$ natural numbers.
Then there is a real number $t \in (0,1)$
depending only on $\Lambda, d,r$ such that if:

$\bullet$ $(X,B)$ is projective lc of dimension $d$,

$\bullet$ the coefficients of $B$ are in $\Lambda$,

$\bullet$ $rM$ is a nef Cartier divisor, and

$\bullet$ $K_X+B+M$ is a big divisor,\\\\
then $K_X+tB+nM$ is a big divisor for any natural number $n>2rd$.\\
\end{lem}
\begin{proof}
Since $M$ is nef, it is enough to treat the case $n=2rd+1$.
We can assume $1\in \Lambda$. Let $(X,B)$ and $M$ be as in the statement of the lemma.
Let $f\colon W\to X$ be a log resolution and let $B_W$ be the sum of the birational transform of
$B$ and the reduced exceptional divisor of $f$, and let $M_W$ be the pullback of $M$.
Then we can replace $(X,B)$ with $(W,B_W)$ and replace $M_W$ with $M$ hence it is enough
to only consider log smooth pairs.

We want to argue that, after extending $\Lambda$ if necessary, it is enough to only consider
the case when $(X,B)$ is klt. If the lemma does not hold, then there is a sequence
$(X_i,B_i)$, $M_i$ of log smooth lc pairs and nef $\Q$-divisors satisfying the
assumptions of the lemma but such that the pseudo-effective thresholds
$$
b_i=\min\{a\ge 0 \mid K_{X_i}+aB_i+nM_i ~~\mbox{is pseudo-effective}\}
$$
is a strictly increasing sequence of numbers approaching $1$. Now by extending $\Lambda$ and decreasing
the coefficients in $B_i$ which are equal to $1$,
we can assume that $(X_i,B_i)$ are klt. To get a contradiction it is obviously
enough to only consider this sequence hence we only need to consider the klt case.

Now let $(X,B)$ and $M$ be as in the statement of the lemma where we assume $(X,B)$ is log smooth klt. 
Let $b$ be the pseudo-effective
threshold as defined above. We may assume $b>0$.  By Lemma \ref{l-LMMP}(2) below, we can run
an LMMP on $K_X+bB+nM$ which ends with a minimal model $X'$ on which $K_{X'}+bB'+nM'$ is semi-ample defining
a contraction $X'\to T'$. Since $b>0$, a general fibre of $X'\to T'$ is positive-dimensional and  
the restriction of $B'$ to it is big,
by the bigness of $K_{X'}+B' + nM'$ and
the definition of $b$.
So relying on Lemma \ref{l-LMMP}(2) once more, we can also run an LMMP$/T'$ on $K_{X'}+nM'$ with scaling of $bB'$
which terminates with a Mori fibre space.
Denote the end result again by $X'$ and the Mori fibre space structure by $X'\to S'$.
By Lemma \ref{l-LMMP}(3), both LMMP's are $M$-trivial and hence the Cartier and nefness of $rM$ is
preserved in the process.
Now since $K_{X'}+bB'+nM'\equiv 0/S'$ and since $n>2rd$,
$M'\equiv 0/S'$ by boundedness of length of extremal rays [\ref{Ka91}].
In particular, if $F'$ is a general fibre of $X'\to S'$,
then
$$
K_{F'}+\Delta':=(K_{X'}+bB')|_{F'} \equiv (K_{X'}+bB' + nM')|_{F'} \sim_{\R} 0 .
$$
By construction, $(K_{X'} + B' + M')|_{F'}$ is big and $M'|_{F'}\equiv 0$, so
$B'|_{F'}$ is not zero and its coefficients belong to $\Lambda$.
Therefore, $b$ is bounded away from $1$ otherwise we get a contradiction with
the ACC property of [\ref{HMX2}, Theorem 1.5]. Thus there is $t_0\in(0,1)$ depending only
on $\Lambda, d,r$ such that $K_X+t_0B+nM$ is pseudo-effective.
Now take $t=\frac{t_0+1}{2}$.\\
\end{proof}

We should point out that although we have used (and continue to use) Lemma \ref{l-LMMP}  but
its proof does not rely on any of the results of this section.

\begin{lem}\label{l-vol-infinity}
Let $X$ be a normal projective variety, $L$ a big $\R$-divisor, and $M$ a nef $\Q$-divisor which is not
numerically trivial. Then
$\vol(L+nM)$ goes to $\infty$ as $n$ goes to $\infty$.
\end{lem}
\begin{proof}
We may write $L\sim_\R A+D$
where $A$ is ample $\Q$-Cartier and $D\ge 0$. Thus
$$
\vol(L+nM)\ge \vol(A+nM)\ge n^\nu A^{d-\nu}\cdot M^\nu
$$
where $d=\dim X$ and $\nu$ is the numerical dimension of $M$. Since $A$ is ample and $\nu>0$,
$A^{d-\nu}\cdot M^\nu>0$.
Hence the above volume goes to infinity as $n$ goes to infinity.\\
\end{proof}

\begin{prop}\label{p-bir-nM-n-large}
Let $\Lambda \subset [0,1]$ be a DCC set of nonnegative real numbers and let $d,r$ be natural numbers.
Then there exists a natural number $m$ depending only on $\Lambda, d,r$ such that if:

$\bullet$ $n$ is a natural number satisfying $n > 2rd$ and $r | n$,

$\bullet$ $(X,B)$ is projective lc of dimension $d$,

$\bullet$ the coefficients of $B$ are in $\Lambda$,

$\bullet$ $rM$ is a nef Cartier divisor, and

$\bullet$ $K_X+B+M$ is a big divisor,\\\\
then $|m(K_X+B+nM)|$ defines a birational map.
\end{prop}
\begin{proof}
\emph{Step 1.}
We prove the proposition by induction on $d$. In particular, we may assume that
the proposition holds in dimension $<d$.

Fix $\beta>0$. Pick $(X,B)$, $M$, and $n$ as in the proposition.
Assume that $\vol(K_X + B + nM) > \beta$.
We first prove the result for such $(X,B)$, $M$, and $n$.
At the end, in Steps 6 and 7, we treat the general case.

As in the proof of Lemma \ref{l-peffth-of-B-in-K+B+nM}, by extending $\Lambda$,
by taking a log resolution of $(X, B)$, and
by decreasing the coefficients of $B$, we can assume that $(X, B)$ is klt.\\

\emph{Step 2.}
By Lemma \ref{l-peffth-of-B-in-K+B+nM},
$K_X + bB + nM$ is big for some $b \in (0, 1)$ depending only on $\Lambda, d, r$.
Thus
$$\begin{aligned}
\vol(K_X + \frac{1}{2}(b+1)B + nM) =
&\vol(\frac{1}{2}(K_X + bB + nM + K_X + B + nM) \\
> &\vol(\frac{1}{2}(K_X + B + nM) > \frac{1}{2^d} \beta .
\end{aligned}$$
Replacing $b$ by $\frac{1}{2}(b+1)$ we may assume that
$$
\vol(K_X + bB + nM) > \beta':=\frac{1}{2^d} \beta .
$$
Moreover, there exists a natural number $p$ depending only on $\Lambda$ and $b$
(and hence only on $\Lambda$, $d, r$) and there exists a boundary $B'$ such that $pB'$
is an integral divisor and $b B \le B'\le B$: this follows from the fact that
we can find $p$ so that $\lambda-b\lambda>\frac{1}{p}$ for every nonzero $\lambda \in \Lambda$
which in turn implies that for each $\lambda$ we can find an integer $0\le i\le p$
such that $b\lambda \le \frac{i}{p}\le \lambda$.
By the calculation above, $\vol(K_X + B' + nM) > \beta'$. Replacing $B$ with $B'$, and
$\beta$ with $\beta'$, we can assume $\Lambda=\{i/p \mid 0\le i\le p\}$ and that $pB$ is integral.\\

\emph{Step 3.}
Applying Lemma \ref{l-LMMP}(2) below, we can replace $X$ with the lc model (=ample model) of $K_X+B+nM$ so that 
we can assume that $K_X+B+nM$ is ample keeping $rM$ nef and Cartier.
Since $\vol(K_X + B + nM) > \beta$, there is a natural number $k > 0$ depending only on $d,\beta$,
such that
$$
\vol(k(K_X + B + nM)) > (2d)^d .
$$
Applying [\ref{HMX2}, Lemma 7.1] to the log pair $(X, B)$ and the big divisor $k(K_X +B + nM)$,
we get a covering family $V \to C$ of subvarieties of $X$ such that if
$x$ and $y$ are two general points of $X$, then we may find $c \in C$ and
$$
0\le \Delta_c \sim_{\R} k(K_X + B + nM)
$$
such that $(X, B + \Delta_c)$ is not klt at $y$ but it is lc at $x$ and there is a unique non-klt place of
$(X, B + \Delta_c)$ whose centre is equal to $V_c$ which contains $x$.\\

\emph{Step 4.}
Let $H := 2k(K_X + B + nM)$.
In this step we make the necessary preparations in order to apply [\ref{HMX2}, Theorem 3.5.4] (=Theorem \ref{t-HMX3.5.4} above).
To do this we need to find a natural number $\gamma$, depending only on $\Lambda, d, r$, and find
a divisor $D$ on the
normalization $W$ of $V_c$ such that $\gamma H |_W - D$ is pseudo-effective and
$|D|$ defines a birational map.

If $\dim W = \dim V_c = 0$, then $\gamma,D$ exist trivially (and $H$ is potentially birational).
So assume that $\dim W \ge 1$.
Now applying the adjunction formula of [\ref{HMX2}, Theorem 4.2] (= Theorem \ref{t-HMX4.2} above)
to the klt pair $(X, B)$ and the
divisor $\Delta_c$, and taking into account Remark \ref{rem-Lambda'},
we can find a boundary $B_W$ on $W$ whose coefficients belong to
a DCC set $\Lambda'$ uniquely determined by $\Lambda,d$, such that the difference
$$
(*) \hskip 1pc (K_X + B + \Delta_c) |_W - (K_W + B_W)
$$
is a pseudo-effective divisor. Further, let $\psi : U \to W$
be a log resolution of $(W, B_W)$
and let $B_U$ be the sum of the strict transform of $B_W$ and the reduced exceptional divisor of $\psi$.
Then
$$
K_U + B_U \ge (K_X + B) |_U .
$$
Denote by $M_U := M|_U$, the pullback of $M$ to $U$,
by the composition
$$U \to W \to V_c \hookrightarrow X$$
which is birational onto its image.
Then
$$
K_U + B_U + M_U \ge (K_X + B + M) |_U .
$$
Hence $K_U + B_U + M_U$  is big because $(K_X + B + M) |_U$ is big
 being the pullback of the big divisor
$K_X + B + M$ to a smooth model of the general subvariety $V_c$.

Since the coefficients of $B_U$ belong to the DCC set $\Lambda'$, since $rM_U$ is a nef Cartier divisor,
and since $n>2rd$, the induction hypothesis implies that
$|m(K_U + B_U + nM_U)|$ defines a birational map for
some $m > 0$ depending only on $\Lambda'$ (and hence on $\Lambda$) and $d, r$.
Thus $|m(K_W + B_W + nM_W)|$ also defines a birational map since it contains the direct image of
$|m(K_U + B_U + nM_U)|$ where $M_W$ denotes the pullback of $M$ to $W$.

Note that the difference
$$\begin{aligned}
&(K_X + B + nM + \Delta_c) |_W - (K_W + B_W + nM_W) \\
\sim_{\R}
&(k+1)(K_X + B + nM)|_W - (K_W + B_W + nM_W)
\end{aligned}$$
is a pseudo-effective divisor by $(*)$ above.
Now  let
$D := m(K_W + B_W + nM_W)$ and let $\gamma$ be the smallest natural number satisfying
$\gamma \ge m(k+1)/2k$. Then $\gamma H |_W - D$ is a pseudo-effective divisor and
$|D|$ defines a birational map as required.\\

\emph{Step 5.}
By Step 4 and Theorem \ref{t-HMX3.5.4},
$$
m'H=2m'k(K_X + B + nM)
$$
is potentially birational for some
$$
m'\le 2(d-1)^2 \gamma + 1 .
$$
Thus by Lemma \ref{l-addNef},
$$
2m'kp(K_X + B + nM)+nM
$$
is also potentially birational and
$$
|K_X + \lceil 2m'kp(K_X + B + n M) + n M \rceil|
$$
defines a birational map where $p$ is as in Step 2 (recall that $pB$ is an integral divisor).
Since
$$
K_X + \lceil 2m'kp(K_X + B + n M) + n M \rceil
\le \lfloor (2m'kp+1)(K_X + B + nM) \rfloor
$$
the linear system
$$
|\lfloor (2m'kp+1)(K_X + B + nM) \rfloor|
$$
also defines a birational map. Now the number $m'':=2m'kp+1$ only depends on the data $\Lambda, d,r,\beta$.\\

\emph{Step 6.}
Now we go back to Step 1.
We will show that there exist a natural number $q$ and a real number $\alpha>0$ depending only on
$\Lambda, d,r$, such that if $(X,B)$, $M$, $n$ are as in the statement of the proposition
and if $n\ge q$, then $\vol(K_X + B + nM)>\alpha$.
 If this is not true, then we can
find a sequence $(X_i,B_i)$, $M_i$, $n_i$ satisfying the assumptions of the proposition
such that the $n_i$ form a strictly increasing sequence approaching $\infty$ and  the
$\vol(K_{X_i} + B_i + n_iM_i)$ approach $0$.
By replacing $X_i$ with a minimal model of $K_{X_i} + B_i + n_iM_i$,
we may assume that $K_{X_i} + B_i + n_iM_i$ is nef.
We can also assume that $\nu$, the numerical dimension of $M_i$, is independent of $i$.
We may assume $\nu>0$ otherwise we can get a contradiction using [\ref{HMX2}, Theorem 1.3].

By Lemma \ref{l-vol-infinity}, for each $i$, there is $n_i'$ the largest natural number divisible by $r$ such
that $\vol(K_{X_i} + B_i + n_i'M_i)<1$.
We show that the volume $\vol(K_{X_i} + B_i + (2n_i'-1)M_i)$ is bounded from above.
This follows from
$$
2^d> \vol(2(K_{X_i}+B_i+n_i'M_i))
$$
$$
=\vol(K_{X_i}+B_i+(2n_i'-1)M_i+K_{X_i}+B_i+M_i)
$$
$$
>\vol(K_{X_i}+B_i+(2n_i'-1)M_i)
$$
where we use the assumption that $K_{X_i}+B_i+M_i$ is big.

On the other hand, since
$$
\vol((K_{X_i} + B_i + (n_i'+r)M_i)) \ge 1,
$$
by Steps 2-5 above,
we may assume that there is an $m''$ depending only on $\Lambda, d,r$ such that
$$
|m''(K_{X_i} + B_i + (n_i'+r)M_i)|
$$
defines a birational map for every $i$. In particular,  there exist resolutions $f_i\colon Y_i\to X_i$
such that 
$$
P_i:=f_i^*m''(K_{X_i} + B_i + (n_i'+r)M_i)\sim H_i+G_i
$$
where $H_i$ is big and base point free and $G_i$ is effective.
So we can calculate
$$
2^d (m'')^d > \vol(m''(K_{X_i}+B_i+(2n_i'-1)M_i))
$$
$$
=(P_i+m''(n_i'-r-1)f_i^*M_i)^d\ge (m''(n_i'-r-1))^\nu H_i^{d-\nu}\cdot f_i^*M_i^\nu
$$
which gives a contradiction as $\lim (n_i'-r-1)=\infty$ and $H_i^{d-\nu}\cdot f_i^*M_i^\nu\ge \frac{1}{r^\nu}$.\\

\emph{Step 7.} Let $q,\alpha$ be as in Step 6. In this step we show that there is $\beta>0$ depending
only on $\Lambda, d,r$ such that $\vol(K_X + B + nM)>\beta$ for any $(X,B)$, $M$, $n$ as in the
statement of the proposition. We may assume $q>n$ otherwise we can use Step 6.
Let $s=\frac{n-1}{q-1}$. Then
$$
\vol(K_X + B + nM)=\vol((1-s)(K_X + B + M)+s(K_X + B + qM))
$$
$$
\ge s^d\vol(K_X + B + qM)>s^d\alpha\ge \frac{\alpha}{(q-1)^d}=:\beta .
$$
This completes the proof of the proposition.\\
\end{proof}

%
%
%

\section{{Generalized polarized pairs}}

In this section, we define generalized lc and klt singularities, discuss some 
of their basic properties, and then define generalized lc thresholds for generalized polarized pairs.
Next we consider running the log minimal model program for these pairs,
and use it to extract divisors with generalized log discrepancy $<1$.
Then we define generalized adjunction and discuss DCC and ACC properties of coefficients 
in the boundary and nef parts of generalized polarized pairs under this adjunction.

\par \vskip 1pc

{\textbf{\sffamily{Generalized singularities.}}}
We already defined generalized polarized pairs in the introduction.
Now we define their singularities.

\begin{defn}\label{d-g-sing}
Let $(X',B'+M')$ be a generalized polarized pair as in \ref{d-g-pol-pair} which comes with the data
 $X \overset{f}\to X'\to Z$ and $M$.
Let $E$ be a prime divisor on some birational model of $X'$.
We define the \emph{generalized log discrepancy} of $E$ with respect to the above generalized polarized pair
as follows. After replacing $X$, we may assume $E$ is a prime divisor on $X$.
We can write
$$
K_X+B+M=f^*(K_{X'}+B'+M')
$$
for some $\R$-divisor $B$.
The generalized log discrepancy of $E$ is defined to be
$1-b$ where $b$ is the coefficient of $E$ in $B$.

We say that $(X',B'+M')$ is \emph{generalized lc} (resp. \emph{generalized klt}) if the generalized
log discrepancy of any
prime divisor is $\ge 0$ (resp. $>0$).
If $f$ is a log resolution of $(X',B')$, then generalized lc (resp. generalized klt) is equivalent
to the coefficients of $B$ being $\le 1$ (resp. $<1$). If the generalized log discrepancy of $E$
is $\le 0$, we call the image of $E$ in $X'$ a \emph{generalized non-klt centre}. If  $(X',B'+M')$ is {generalized lc},
a non-klt centre is also referred to as a \emph{generalized lc centre}.
\end{defn}

\begin{rem}\label{r-g-sing}
We use the notation of \ref{d-g-sing}.

(1)
Note that $Z$ does not play any role in the definition of singularities.
That is because singularities are local in nature over $X'$, so one can simply
 assume $X'\to Z$ is the identity map. The same applies to generalized
lc thresholds defined below (\ref{d-g-lct}) and in general to notions and 
statements that are local.

(2)
Assume that  $(X',B'+M')$ is generalized klt. Let $D'$ be an effective $\R$-Cartier divisor.
Then from the definitions we can easily see that $(X',B'+\epsilon D'+ M')$ is generalized
klt with boundary part $B'+\epsilon D'$ and nef part $M$, for any small $\epsilon >0$.

Now assume that $D'$ is ample$/Z$. Then for any $a>0$ we can find a boundary
$$
\Delta'\sim_\R B'+a D'+ M'/Z
$$
such that $(X',\Delta')$ is klt.

(3)
Assume that $K_{X'}+B'$ is $\R$-Cartier and write $K_X+\tilde{B}=f^*(K_{X'}+B')$
and $f^*M'=M+E$. By the negativity lemma [\ref{Shokurov-log-flips}, Lemma 1.1],
$E\ge 0$. Thus $B=\tilde B+E\ge \tilde B$. Therefore, if $(X',B'+M')$ is generalized
lc (resp. generalized klt), then $(X',B')$ is lc (resp. klt).

(4)
Assume that $M\sim_\R 0/X'$. Then $(X',B'+M')$ is generalized lc (resp. generalized klt) iff $(X',B')$ is
generalized lc (resp. generalized klt). Indeed in this case $M=f^*M'$ hence
$K_X+B=f^*(K_{X'}+B')$ which implies the claim. In this situation $M'$ does not
contribute to the singularities even if its coefficients are large. In contrast, the larger
the coefficients of $B$, the worse the singularities.

(5)
In general, $M$ does contribute to singularities. For example, assume $X'=\PP^2$ and that $f$ is
the blowup of a point $x'$. Let $E$ be the exceptional divisor, $L'$ a line passing through
$x'$ and $L$ the birational transform of $L'$.

If $B'=0$ and $M=2L$, then
we can calculate $B=E$ hence $(X',B'+M')$ is generalized lc but not generalized klt.
However, if $B'=L'$ and $M=2L$, then
$(X',B'+M')$ is not generalized lc because in this case $B=L+2E$.

(6)
Assume we are given a contraction  $X'\to Y/Z$. We may assume
$f$ is a log resolution of $(X',B')$. Let $F$ be a general fibre of
$X\to Y$, $F'$ the corresponding fibre of $X'\to Y$, and $g\colon F\to F'$ the induced morphism.
Let
$$
B_F=B|_F, \,\, M_F=M|_F, \,\, B_{F'}=g_*B_F, \,\, M_{F'}=g_*M_F.
$$
Then
$(F',B_{F'}+M_{F'})$ is a generalized polarized pair with the data $F\to F'\to Z$ and $M_F$.
Moreover,
$$
K_{F'}+B_{F'}+M_{F'}=(K_{X'}+B'+M')|_{F'} .
$$
In addition, $B_{F'}=B'|_{F'}$ and $M_{F'}=M'|_{F'}$: note that
since $F'$ is a general fibre, $B'$ and $M'$ are $\R$-Cartier along any codimension
one point of $F'$ hence we can define these restrictions.

(7)
Let $\phi\colon X''\to X'$ be a birational contraction from a normal variety.
We can assume $X\bir X''$ is a morphism. Let $B'',M''$ be the pushdowns of $B,M$.
Then
$$
K_{X''}+B''+M''=\phi^*(K_{X'}+B'+M').
$$
Now assume that $B''$ is a boundary. Then we can naturally consider
$(X'',B''+M'')$ as a generalized polarized pair with boundary part $B''$
and nef part $M$. One may think of $({X''},B''+M'')$ as a crepant model of $({X'},B'+M')$.\\
\end{rem}

\begin{defn}\label{d-g-lct}
Let $(X',B'+M')$ be a generalized polarized pair as in \ref{d-g-pol-pair} which
comes with the data $X \overset{f}\to X'\to Z$ and $M$.
Assume that $D'$ on $X'$ is an effective $\R$-divisor and that $N$ on $X$ is an $\R$-divisor which is
nef$/Z$ and that $D'+N'$ is $\R$-Cartier.
The \emph{generalized lc threshold} of $D'+N'$
with respect to $(X',B'+M')$ (more precisely, with respect to the above data)
is defined as
$$
\sup \{s \mid \mbox{$(X',B'+sD'+M'+sN')$ is generalized lc}\}
$$
where the pair in the definition has boundary part $B'+sD'$ and nef part $M+sN$.

By the negativity lemma, $G:=f^*(D'+N')-N\ge 0$. Thus we can write
$$
K_X+B+M=f^*(K_{X'}+B'+M')
$$
and
$$
K_X+B+sG+M+sN=f^*(K_{X'}+B'+sD'+M'+sN') .
$$
In particular, if $(X',B'+M')$ is generalized lc, then
the just defined generalized lc threshold is nonnegative. However, the threshold might be $+\infty$:
this happens when $D'=0$ and $N\sim_\R 0/X'$.
\end{defn}

As pointed earlier, the generalized lc threshold is local over $X'$, so we can usually
assume $X'\to Z$ is the identity map.
When $M=N=0$, we recover the usual lc threshold of $D'$ with respect to $(X',B')$.\\

{\textbf{\sffamily{LMMP for generalized polarized pairs.}}}
Let $(X', B'+M')$ be a $\Q$-factorial generalized lc polarized pair with data $X \overset{f}\to X'\to Z$ and $M$.
One can ask whether one can run an LMMP$/Z$ on $K_{X'}+B'+M'$ and whether it terminates.
We cannot answer this question in such generality but we will put some extra assumptions 
under which the answer would be yes.

Assume that $K_{X'}+B'+M'+A'$ is nef$/Z$
for some $\R$-Cartier divisor $A'\ge 0$ which is big$/Z$. Moreover, assume

\par \vskip 1pc
$(*)~~~$ for any $s\in (0,1)$ there is  a boundary $\Delta'\sim_\R B'+sA'+M'/Z$
such that $(X',\Delta'+(1-s)A')$ is klt.

\par \vskip 1pc
Condition $(*)$ is
automatically satisfied if $A'$ is general ample$/Z$ and either

\par \vskip 1pc
\begin{itemize}
\item[(i)]
$(X',B'+M')$ is generalized klt, or

\item[(ii)]
$(X',B'+M')$ is generalized lc and $(X',0)$ is klt.

\end{itemize}

\par \vskip 1pc
We will show that we can run the LMMP$/Z$ on $K_{X'}+B'+M'$ with scaling of $A'$
(However, we do not know  whether it terminates).
Let
$$
\lambda=\min \{t\ge 0\mid K_{X'}+B'+M'+tA'~~\mbox{is nef$/Z$}\} .
$$
We may assume $\lambda>0$. Replacing $A'$ with $\lambda A'$ we may assume $\lambda=1$.
By assumption we can find a
number $0<s<1$ and a boundary $\Delta'\sim_\R B'+sA'+M'/Z$
such that $(X',\Delta'+(1-s)A')$ is klt. Now by [\ref{Bi-2}, Lemma 3.1], there is an
extremal ray $R'/Z$ such that $(K_{X'}+\Delta')\cdot R'<0$ and
$$
(K_{X'}+\Delta'+(1-s)A')\cdot R'=0 .
$$
In particular, $(K_{X'}+B'+M')\cdot R'<0$ and
$$
(K_{X'}+B'+M'+ A')\cdot R'=0 .
$$
Moreover, $R'$ can be contracted and its flip exists if it is of flipping type.
If $R'$ defines a Mori fibre space we stop. Otherwise let $X'\bir X''$ be the
divisorial contraction or the flip of $R'$.

Replacing $X$ we may assume $X\bir X''$ is a morphism. Then $(X'',B''+M'')$
is naturally a generalized lc polarized pair with boundary part $B''$ and nef part $M$.
Moreover, $K_{X''}+B''+M''+A''$ is nef$/Z$ and $(*)$ is preserved.
Repeating the process gives the LMMP.\\

Now we show the LMMP terminates under suitable assumptions.

\begin{lem}\label{l-LMMP}
Let $(X', B'+M')$ be a $\Q$-factorial generalized lc polarized pair of dimension $d$ with data
 $X \overset{f}\to X'\to Z$ and $M$. Assume that $(X', B'+M')$  satisfies $({\rm i})$ or $({\rm ii})$ above.
Run an LMMP$/Z$ on $K_{X'}+B'+M'$ with scaling of some
general ample$/Z$ $\, \R$-Cartier divisor $A'\ge 0$. Then the following hold:

\begin{itemize}
\item[(1)]
Assume that $K_{X'}+B'+M'$ is not pseudo-effective$/Z$. Then the LMMP terminates with a
Mori fibre space.

\item[(2)]
Assume that

$\bullet$ $K_{X'}+B'+M'$ is pseudo-effective$/Z$,

$\bullet$ $(X', B'+M')$ is generalized klt, and that

$\bullet$ $K_{X'}+(1+\alpha)B'+(1+\beta)M'$ is $\R$-Cartier and big$/Z$
for some $\alpha,\beta\ge 0$.

Then the LMMP terminates with a minimal model $X''$ and $K_{X''}+B''+M''$ is semi-ample$/Z$,
hence it defines a contraction $\phi\colon X''\to T''/Z$. If moreover
a general fibre of $\phi$ is positive-dimensional and if the restriction of $B''$ to it
is nonzero,
then we can
run the LMMP$/T''$ on $K_{X''}+M''$ with scaling of $B''$ which terminates with a Mori fibre space of 
$K_{X''}+M''$ over both $T''$ and $Z$.

\item[(3)]
Assume $X\to X'$ is the identity morphism and that $M=\sum \mu_j M_j$ where $\mu_j\ge 0$ and
$M_j$ are Cartier nef$/Z$ divisors. Pick $j$ and assume $\mu_j>2d$. Then the above LMMP's are $M_j'$-trivial.
In particular, the LMMP's preserve the Cartier and the nefness$/Z$ of $M_j'$. Moreover, under the assumptions of
$({\rm 2})$ and assuming $\phi$ is birational,  $M_j''\equiv 0/T''$ and $M_j''$ is the pullback of some Cartier divisor on $T''$.
\end{itemize}
\end{lem}
\begin{proof}
(1)
Since $K_{X'}+B'+M'$ is not pseudo-effective$/Z$,
the LMMP is also an LMMP on $K_{X'}+B'+\epsilon A'+M'$ with scaling of $(1-\epsilon)A'$
for some $\epsilon>0$. Now we can find a boundary
$$
\Delta'\sim_\R B'+\epsilon A'+M'/Z
$$
such that $(X',\Delta'+(1-\epsilon)A')$ is klt. The claim then follows from [\ref{BCHM}] as 
the LMMP is an LMMP$/Z$ on $K_{X'}+\Delta'$ with scaling of $(1-\epsilon)A'$.

(2)
As
$$
K_{X'}+(1+\alpha)B'+(1+\beta)M'
$$
is big$/Z$, it is $\R$-linearly equivalent to some
$P'+G'$ over $Z$ where $P'$ is ample and $G'\ge 0$. Now if $\epsilon>0$ is small, then
$$
(1+\epsilon)(K_{X'}+B'+M')\sim_\R K_{X'}+(1-\epsilon\alpha)B'+(1-\epsilon\beta)M'+\epsilon P'+\epsilon G'
$$
$$
\sim_\R K_{X'}+\Delta'/Z
$$
for some $\Delta'$ such that $(X',\Delta')$ is klt and $\Delta'$ is big$/Z$.
The LMMP is also an LMMP$/Z$ on $K_{X'}+\Delta'$ with scaling of $(1+\epsilon) A'$ which terminates
on some model $X''$ by [\ref{BCHM}]. By the base point free theorem
for klt pairs with big boundary divisor [\ref{BCHM}, Corollary 3.9.2],
$K_{X''}+\Delta''$ is semi-ample$/Z$ hence $K_{X''}+B''+M''$ is semi-ample$/Z$ and so it defines a
contraction $\phi\colon X''\to T''$.

Now assume
a general fibre of $\phi: X'' \to T''$ is positive-dimensional and the restriction of $B''$ to it
is nonzero. In particular, this implies that $K_{X''}+M''$ is not pseudo-effective over $T''$. 
Since 
$$
\frac{1}{1+\epsilon} (K_{X''}+\Delta'')\equiv K_{X''}+B''+M''\equiv 0/T'',
$$ 
running the LMMP$/T''$ on
$K_{X''}+M''$ with scaling of $B''$ is the same as running the LMMP$/T''$
on $K_{X''}+\Delta''-\tau B''$ with scaling of $\tau B''$ for some small $\tau>0$ and this
terminates with a Mori fibre space over $T''$ and also over $Z$, by [\ref{BCHM}]. Note that, $\Delta''-\tau B''\ge 0$ by construction. 

(3)
Each step of those LMMP's is $M_j'$-trivial and preserves the Cartier and the nefness$/Z$
of $M_j'$ by boundedness of the length of extremal rays and the
cone theorem [\ref{Ka91}][\ref{KM}, Theorem 3.7 (1) and (4)]. Under the assumptions of
(2) and assuming $\phi$ is birational, to show that $M_j'$ is the pullback of some Cartier divisor on $T''$, it is enough to
show that $X''\to T''$ decomposes into a sequence of extremal contractions which are
negative with respect to certain klt pairs. We write this more precisely.

Since $\Delta'$ in the proof of (2) is big$/Z$, we can assume $\Delta''\ge C''$
for some ample $\Q$-divisor $C''$. Since $K_{X''}+\Delta''\equiv 0/T''$,
if $X''\to T''$ is not an isomorphism, then there is a
$(K_{X''}+\Delta''-C'')$-negative extremal ray which gives a contraction $X''\to X_2''/T''$.
In particular $M_j''$ is the pullback of a Cartier divisor on $X_2''$ [\ref{KM}, Theorem 3.7 (4)].
Now repeat the process with $X_2''$ and so on. Since $\phi$ is birational by assumption,
the process ends with $T''$ hence we can indeed decompose $X''\to T''$ into a sequence of extremal
contractions as required.\\
\end{proof}

We will apply the LMMP to birationally extract certain divisors for a generalized polarized pair.

\begin{lem}\label{l-extract-divs-1}
Let $(X',B'+M')$ be a generalized lc polarized pair with data $X \overset{f}\to X'\to Z$ and $M$.
Let $S_1, \dots, S_r$ be prime divisors on birational models of $X'$ which are exceptional$/X'$ and
whose generalized log discrepancies with respect to $(X',B'+M')$ are at most $1$. Then perhaps 
after replacing $f$ with a high resolution, there exist a $\Q$-factorial generalized lc polarized pair $(X'',B''+M'')$ 
with data $X \overset{g}\to X''\to Z$ and $M$, and a projective birational morphism 
$\phi\colon X''\to X'$ such that 

$\bullet$ $S_1, \dots, S_r$  appear as divisors on $X''$,

$\bullet$ each exceptional divisor of $\phi$ is one of the
$S_i$ or is a component of $\rddown{B''}$, and 

$\bullet$ $K_{X''}+B''+M''=\phi^*(K_{X'}+B'+M')$. 

In particular, the exceptional divisors
of $\phi$ are exactly the $S_i$ if $(X',B'+M')$ is generalized klt.
\end{lem}
\begin{proof}
Replacing $X$ we may assume the $S_i$ are divisors on $X$
and that $f$ is a log resolution of $(X',B')$. Let $E_1, E_2, \dots $ be the exceptional
divisors of $f$ where we can assume $E_i=S_i$ for $i\le r$.
Write
$$
K_X+B+M=f^*(K_{X'}+B'+M')
$$
and let $\Delta=B+E$ where $E:=\sum_{i>r} a_iE_i$ and $a_i$ is the generalized log discrepancy
of $E_i$ (by definition $a_i$ is equal to $1-b_i$ where $b_i$ is the coefficient of $E_i$ in $B$).
Then $\Delta$ is a boundary and
$$
K_X+\Delta+M=f^*(K_{X'}+B'+M')+E\equiv E/X'
$$
with $E\ge 0$ exceptional$/X'$. By construction, none of the $S_i$ are components of $E$.

Now  run an LMMP$/X'$ on $K_X+\Delta+M$ with scaling of some ample divisor.
This is also an LMMP$/X'$ on $E$. In the course of the LMMP we arrive at a model $X''$ on which
$K_{X''}+\Delta''+M''$ is a limit of movable$/X'$ divisors hence it is nef on the general curves$/X'$ 
of any exceptional divisor of $X''\to X'$
where $\Delta'', M''$ are the pushdowns of $\Delta,M$.
But since $E''$ is effective and exceptional$/X'$, $E''=0$ by the general negativity lemma
(cf. [\ref{Bi}, Lemma 3.3 and the proof of Theorem 3.4]).

Note that since the LMMP contracts $E$, we have $\Delta''=B''$. So we can write
$$
K_{X''}+B''+M''=\phi^*(K_{X'}+B'+M')
$$
where $\phi$ is the morphism $X''\to X'$. By construction, none of the $S_i$ is
contracted by the LMMP. Moreover, any exceptional divisor of $\phi$ is one of the
$S_i$ or is a component of $\rddown{B''}$. In particular, the exceptional divisors
of $\phi$ are exactly the $S_i$ if $(X',B'+M')$ is generalized klt. Note that 
$X''$ is $\Q$-factorial by construction.

\end{proof}

\begin{lem}\label{l-extract-div-2}
Under the notation and assumptions of Lemma \ref{l-extract-divs-1}, further assume 
 that $(X',C')$ is klt for some $C'$, and that the generalized log discrepancies of the 
 $S_i$ with respect to $(X',B'+M')$ are $<1$. Then we can construct $\phi$ so that 
in addition it satisfies: 
 
$\bullet$ its exceptional divisors are exactly $S_1,\dots, S_r$, and 

$\bullet$ if $r=1$ and $X'$ is $\Q$-factorial, then $\phi$ is an extremal contraction. 
\end{lem}
\begin{proof}
Since $(X',C')$ is klt and $(X',B'+M')$ is generalized lc,  
$$
(X',(1-\epsilon) B'+\epsilon C'+(1-\epsilon)M')
$$ 
is generalized klt for any 
small $\epsilon>0$ with boundary part $\Gamma':=(1-\epsilon) B'+\epsilon C'$ and 
nef part $(1-\epsilon)M$. Moreover, the generalized  log discrepancies 
of the $S_i$ with respect to $(X',\Gamma'+(1-\epsilon) M')$ are still less than 
$1$. So by Lemma \ref{l-extract-divs-1}, there is $\phi\colon X''\to X'$ which extracts exactly the $S_i$. 

Now further assume that $r=1$ and that $X'$ is $\Q$-factorial. 
 By construction, we can write   
$$
K_{X''}+\Gamma''+(1-\epsilon)M''=\phi^*(K_{X'}+\Gamma'+(1-\epsilon) M')
$$ 
where $\Gamma''$ is the sum of the 
birational transform of $\Gamma'$ and $sS_1''$ for some $s\in (0,1)$. Now run an 
LMMP$/X'$ on $K_{X''}+\Gamma''+\delta S_1''+(1-\epsilon)M''$ for some small $\delta>0$ which is 
also an LMMP on $S''_1$.  
Since $X'$ is $\Q$-factorial, the last step 
of the LMMP is an extremal contraction $X'''\to X'$ which contracts $S_1'''$, the pushdown of $S_1''$, 
and $X''\bir X'''$ is an isomorphism in codimension one. 
Thus replacing $X''$ with $X'''$ we can assume $\phi$ is extremal.\\
\end{proof}

{\textbf{\sffamily{Generalized adjunction.}}}
We define an adjunction formula for generalized polarized pairs similar to the traditional one.

\begin{defn}\label{d-q-adjunction}
Let $(X',B'+M')$ be a generalized polarized pair with data $X \overset{f}\to X'\to Z$ and $M$.
Assume that $S'$ is the normalization of a component of $\rddown{B'}$ and $S$ is its birational transform on $X$.
Replacing $X$ we may assume
$f$ is a log resolution of $(X',B')$.
Write
$$
K_X+B+M=f^*(K_{X'}+B'+M')
$$
and let
$$
K_S+B_S+M_S:=(K_X+B+M)|_S
$$
where $B_S=(B-S)|_S$ and $M_S=M|_S$.
Let $g$ be the induced morphism $S\to S'$ and let $B_{S'}=g_*B_S$ and $M_{S'}=g_*M_S$.
Then we get the equality
$$
K_{S'}+B_{S'}+M_{S'}=(K_{X'}+B'+M')|_{S'}
$$
which we refer to as \emph{generalized adjunction}. It is obvious that $B_{S'}$ depends on both $B'$ and $M$.

Now assume that $(X',B'+M')$ is generalized lc.
By Remark \ref{r-adjunction} below $B_{S'}$ is a boundary divisor on $S'$, i.e. its coefficients
belong to $[0,1]$. We consider $(S',B_{S'}+M_{S'})$ as
a generalized polarized pair which is determined by the boundary part $B_{S'}$, the morphisms
$S\to S'\to Z$, and the nef part $M_S$.
It is also clear that $(S',B_{S'}+M_{S'})$ is generalized lc if $(X',B'+M')$ is so because
then
$$
K_{S}+B_{S}+M_{S}=g^*(K_{S'}+B_{S'}+M_{S'})
$$
and the coefficients of $B_S$ are at most $1$.\\
\end{defn}

\begin{rem}\label{r-adjunction}
We will argue that the $B_{S'}$ defined in \ref{d-q-adjunction} is indeed a boundary divisor on $S'$,
if $(X',B'+M')$ is generalized lc. The lc property immediately implies that
the coefficients of $B_{S'}$ do not exceed $1$, hence we only have to show that $B_{S'}\ge 0$.
Moreover, if $K_{X'}+B'$ is $\R$-Cartier, then  $B_{S'}\ge 0$ follows from the usual
divisorial adjunction: indeed in this case if $\tilde{B}_{S'}$ is the divisor given by the adjunction
$$
K_{S'}+\tilde{B}_{S'}=(K_{X'}+B')|_{S'}
$$
then it is well-known that $\tilde{B}_{S'}$ is a boundary divisor, and it is also
clear from our definitions that $\tilde{B}_{S'}\le B_{S'}$.

In practice when we apply generalized adjunction, $X'$ will be $\Q$-factorial, hence
$K_{X'}+B'$ will be $\R$-Cartier.
But for the sake of completeness we treat the general case, i.e. the non-$\R$-Cartier $K_{X'}+B'$ case.
We will reduce the statement to the situation $\dim X'=2$ in which case $K_{X'}+B'$ turns out to be
$\R$-Cartier automatically. Assume $\dim X'>2$.
Let $H'$ be a general hypersurface section and $G'$ its pullback to $S'$.
Adding $H'$ to $B'$ we may assume $H'$ is a component of $\rddown{B'}$.
Both $H'$ and $G'$ are normal varieties. Let $B_{H'}$ be given by the generalized
adjunction
$$
K_{H'}+B_{H'}+M_{H'}=(K_{X'}+B'+M')|_{H'} .
$$
Since $H'$ is a general hypersurface section, $B_{H'}$ is simply the intersection of
$B'-H'$ with $H'$, that is, each component of $B_{H'}$ is a component of the intersection of
some component of $B'-H'$ with $H'$ inheriting the same coefficient.
In particular, $B_{H'}$ is a boundary divisor and $G'$ is a component of $\rddown{B_{H'}}$.

A further generalized adjunction and induction on dimension gives
$$
K_{G'}+B_{G'}+M_{G'}=(K_{H'}+B_{H'}+M_{H'})|_{G'}
$$
where $B_{G'}$ is a boundary. But $B_{G'}$ is equal to the intersection of
$B_{S'}-G'$ with the ample divisor $G'$ on $S'$ which implies that $B_{S'}$ is a boundary divisor too.

Now we can assume $\dim X'=2$. Since
$$
K_{X}+B+M=f^*(K_{X'}+B'+M')\equiv 0/X'
$$
and since $M$ is nef$/X'$, there is a divisor $\tilde{B}\le B$
such that $K_X+\tilde{B}\equiv 0/X'$ and $f_*\tilde{B}=B'$.
Since each coefficient of $B$ is at most $1$, each coefficient of $\tilde{B}$
is also at most $1$. Therefore $(X',B')$ is \emph{numerically lc} (see [\ref{KM}, Section 4.1];
note however that [\ref{KM}] only considers $B'$ with rational coefficients but
all the definitions and results that we need make sense and hold true for real coefficients as well).
Now by [\ref{KM}, Section 4.1], $(X',B')$ is lc. In particular, $K_{X'}+B'$ is $\R$-Cartier.
So we are done by the above arguments.\\
\end{rem}

\begin{prop}\label{p-adj-dcc}
Let $d$ be a natural number and $\Lambda$ a DCC set of nonnegative real numbers.
Then there is a DCC set $\Omega$ of nonnegative real numbers
depending only on $d$ and $\Lambda$ such that if
$(X',B'+M')$ is a generalized lc polarized pair
of dimension $d$ with data $X \overset{f}\to X'\to Z$ and $M$,
and $S'$ is the normalization of a component of $\rddown{B'}$ satisfying

$\bullet$ $M=\sum \mu_jM_j$ where $M_j$ are nef$/Z$ Cartier divisors and $\mu_j\in \Lambda$,

$\bullet$ the coefficients of $B'$ belong to $\Lambda$, and

$\bullet$ $B_{S'}$ is given by the following generalized adjunction (as in \ref{d-q-adjunction})
$$
K_{S'}+B_{S'}+M_{S'}=(K_{X'}+B'+M')|_{S'}, $$
then the coefficients of $B_{S'}$ belong to $\Omega$.
\end{prop}

\begin{proof}
If the statement does not hold, then there exist a sequence of generalized lc polarized pairs
$(X_i',B_{i}'+M_{i}')$ and $S_i'$, with data $X_i \overset{f_i}\to X_i'\to Z_i$ and $M_i=\sum \mu_{j,i}M_{j,i}$,
 satisfying the assumptions of the proposition but such that the set of the coefficients of
all the $B_{S_i'}$ put together does not satisfy DCC. Note that since the problem 
is local, we may assume $X_i'\to Z_i$ is the identity map for each $i$. 
 We may also assume $f_i$ is a log resolution of $(X',B')$.

Let $S_i\subset X_i$ be the birational
transform of $S_i'$.
We can assume that each $B_{S_i'}$ has a component $V_i$ with
coefficient $a_i$
such that $\{a_i\}$ is a strictly decreasing sequence. Let $a=\lim a_i$.

We may assume that the $K_{X_i'}+B_i'$ are $\R$-Cartier otherwise as in Remark \ref{r-adjunction},
by taking hypesurface sections, we reduce the problem to dimension $2$ in which case
this $\R$-Cartier property holds automatically.
Let $\tilde{B}_{S_i'}$ be the divisor given by the adjunction
$$
K_{S_i'}+\tilde{B}_{S_i'}=(K_{X_i'}+B_{i}')|_{S_i'} .
$$
It is clear from our definitions that $\tilde{B}_{S_i'}\le B_{S_i'}$.
If $c_i$ is the coefficient of $V_i$ in $\tilde{B}_{S_i'}$, then we may assume $c_i\le a_i \le a+\epsilon$ for some fixed
$\epsilon>0$ so that $a+\epsilon<1$. Therefore, $(X_i',B_{i}')$ is plt near
the generic point of (the image of) $V_i$ (this follows from inversion of adjunction
on surfaces [\ref{Shokurov-log-flips}, Corollary 3.12])
and there is a natural number $l$ depending only on $a+\epsilon$ such that 
for each $i$ there is $l_i\le l$ so that 
for any Weil divisor $D_{i}'$ on $X_i'$ the divisor
$l_iD_{i}'$ is Cartier
near the (image of the) generic point of $V_i$ [\ref{Shokurov-log-flips}, Proposition 3.9].
Moreover, by [\ref{Shokurov-log-flips}, Corollary 3.10] we can write
$$
c_i=\frac{l_i-1}{l_i}+\sum b_{k,i}\frac{d_{k,i}}{l_i}
$$
for some nonnegative integers $d_{k,i}$
where $b_{k,i}$ are the coefficients of the
components of $B_{i}'$ other than (the image of) $S_i'$
passing through $V_i$.

On the other hand, shrinking $X_i'$ if necessary we can assume $M_{j,i}'$ is $\Q$-Cartier 
for each $j,i$ so we can write
$$
f_i^*M_{j,i}'=M_{j,i}+E_{j,i}
$$
where the exceptional divisor $E_{j, i}$ is effective by the negativity lemma.
Since $l_iM_{j,i}'$
is Cartier near the (image of the) generic point of $V_i$,
the multiplicity of the birational transform of $V_i$ in $E_{j,i}|_{S_i}$ is equal to
$\frac{e_{j,i}}{l_i}$ for some nonnegative integer $e_{j,i}$.
Therefore,
$$
a_i=\frac{l_i-1}{l_i}+\sum b_{k,i}\frac{d_{k,i}}{l_i}+\sum \mu_{j,i}\frac{e_{j,i}}{l_i} .
$$
This is a contradiction because the above expression and Lemma \ref{l-ACC-DCC} show that the
set $\{a_i\}$ satisfies DCC,
while the $a_i$ form a strictly decreasing sequence. \\
\end{proof}

We will need the next technical lemma in the proof of Proposition \ref{p-global-acc-glc} 
to treat Theorem \ref{t-global-acc} inductively.

\begin{lem}\label{l-adjunction-2}
Let $d$ be a natural number and $\Lambda$ be a DCC set of nonnegative real numbers. 
Let $(X_i',B_{i}'+M_{i}')$ be a sequence of generalized lc polarized pairs of dimension $d$ 
with data $X_i\to X_i'\to Z_i$ and
$M_i$. Let $S_i'$ be the normalization of a component of $\rddown{B_i'}$ and consider 
the generalized adjunction formula
$$
K_{S_i'}+{B}_{S_i'}+M_{S_i'}=(K_{X_i'}+B_{i}'+M_i')|_{S_i'} .
$$ 
Assume further that
 
\begin{enumerate}
\item $X_i'$ is $\Q$-factorial and $Z_i$ is a point,

\item $B_i'=\sum b_{k,i}B_{k,i}'$ where $B_{k,i}'$ are distinct prime divisors and $b_{k,i}\in\Lambda$,

\item $M_i=\sum \mu_{j,i}M_{j,i}$ where $M_{j,i}$ are nef Cartier divisors and $\mu_{j,i}\in\Lambda$,

\item and one of the following holds: 

$\rm (i)$ $\{b_{1,i}\}$ is not finite, and $B_{1,i}'|_{S_i'}\not\equiv 0$ for each $i$, or 

$\rm (ii)$ $\{\mu_{1,i}\}$ is not finite, and $M_{1,i}'|_{S_i'}\not\equiv 0$ for each $i$.\\
\end{enumerate}
Then the set of the coefficients of all the ${B}_{S_i'}$ union the set
$\{\mu_{j,i} \mid M_{j,i}|_{S_i}\not\equiv  0\}$ is not finite.
\end{lem}
\begin{proof}
Let $V_i$ be a prime divisor on $S_i'$. As in the proof of Proposition \ref{p-adj-dcc},
the coefficient of $V_i$ in $B_{S_i'}$ is of the form
$$
a_i=\frac{l_i-1}{l_i}+\sum b_{k,i}\frac{d_{k,i}}{l_i}+\sum \mu_{j,i}\frac{e_{j,i}}{l_i}
$$
where $l_i$ is a natural number and $d_{k,i},e_{j,i}$ are nonnegative integers which are 
contributed by the $B_{k,i}'$ and $M_{j,i}'$ respectively.

Now assume $\rm (i)$ of (4) holds. Since $\{b_{1,i}\}$ is not finite, we can assume 
$b_{1,i}<1$ for each $i$ which in particular means $B_{1,i}'$ is not equal to the image of $S_i'$.
Thus $B_{1,i}'|_{S_i'}$ is a nonzero effective divisor for each $i$.
Choose $V_i$ to be a component of $B_{1,i}'|_{S_i'}$.
Then the set $\{a_i\}$ cannot be finite by Lemma \ref{l-ACC-DCC}
because $\{b_{1,i}\}$ is not finite and $d_{1,i}$ is positive. 

Next assume $\rm (ii)$ of (4) holds. 
Although  $M_{1,i}'|_{S_i'}$ is not numerically trivial by assumption but 
$M_{1,i}|_{S_i}$ may be numerically trivial for some $i$. 
If $M_{1,i}|_{S_i}$ is not numerically trivial
for infinitely many $i$, then obviously the set $\{\mu_{j,i} \mid M_{j,i}|_{S_i}\not\equiv  0\}$ is not finite 
and we are done.
So we may assume $M_{1,i}|_{S_i}$ is numerically trivial for every $i$. Recall 
from the proof of Proposition \ref{p-adj-dcc} that we can assume $f_i^*M_{j,i}'=M_{j,i}+E_{j,i}$ with 
$E_{j,i}\ge 0$. Now we can choose
$V_i$ so that $e_{1,i}\neq 0$ for each $i$: indeed since $M_{1,i}'|_{S_i'}\not\equiv 0$ but $M_{1,i}|_{S_i}\equiv 0$,
we deduce that $E_{1,i}|_{S_i}\neq 0$ and that its pushdown to $S_i'$ is also not
zero; thus the components of the pushdown of $E_{1,i}|_{S_i}$ are components of $B_{S_i'}$,
hence we can choose $V_i$ to be
one of these components. Again this shows that $\{a_i\}$ cannot be finite because $\{\mu_{1,i}\}$ is not finite
and $e_{1,i}>0$.\\
\end{proof}


\section{Bounds on the number of coefficients of $B_i'$ and $M_i'$}\label{s-bnd-comp}

A well-known fact says that if $(X,B)$ is a lc pair, then near each point $x\in X$ the number of components of $B$ 
with coefficient $\ge b>0$ is bounded in terms of $b$ and dimension of $X$.  There is also a global 
version of this fact. In this section, we prove similar local and global statements bounding the number of the 
coefficients of $B_i'$
and the $\mu_j$ in $M=\sum \mu_j M_j$ of a generalized lc polarized pair $(X',B'+M')$ under certain assumptions. 
These bounds will be used in the proof of Proposition \ref{p-glct-to-global-acc}.

We start with a global statement for pairs which can also be applied to generalized polarized pairs. 

\begin{prop}\label{p-toric}
Let $d$ be a natural number and $b$ a positive real number.
Let $(X,B)$ be a projective lc pair of dimension $d$
such that
\begin{itemize}
\item[(i)]
$B\ge\sum_1^r B_k$ where $B_k\ge 0$ are big $\R$-Cartier divisors,
\item[(ii)]
$B_k = \sum b_{j,k} B_{j,k}$ is the irreducible decomposition and $b_{j,k} \ge b$ for every $j,k$, and
\item[(iii)]
$K_{X}+B+P\equiv 0$ for some pseudo-effective $\R$-Cartier divisor $P$.\\
\end{itemize}
Then the number of the $B_k$ is at most $(d+1)/b$, that is, $r\le (d+1)/b$.
\end{prop}

\begin{proof}
Let $(Y,\Delta)$ be a $\Q$-factorial dlt model of $(X, B-\sum_1^r B_k)$ and $f\colon Y\to X$ the
corresponding morphism. By definition, $\Delta$ is the sum of the reduced exceptional
divisor of $f$ and the birational transform of $B-\sum_1^r B_k$. Moreover, since $(X,B)$ is lc, 
$\Supp (\sum_1^r B_k)$ does not contain the image of
any exceptional divisor of $f$, hence $f^*B_k$ is equal to the birational transform of
$B_k$. In particular, $f^*B_k$ is big and it inherits the same coefficients as $B_k$.
Moreover, by letting $B_Y:=\Delta+\sum_1^r f^*B_k$ we get 
$$
K_Y+B_Y+f^*P=K_Y+\Delta+\sum_1^r f^*B_k+f^*P=f^*(K_X+B+P)\equiv 0 .
$$
Now by replacing $(X,B)$ with $(Y,B_Y)$ and
replacing $P$ with $f^*P$ we can assume that $(X,0)$ is $\Q$-factorial klt.
Moreover, by adding $B-\sum_1^r B_k$ to $P$ we can assume $B=\sum_1^r B_k$.

If $P \not\equiv 0$, then $K_X+B$ is not pseudo-effective so we can run an
LMMP on $K_{X}+B$ which terminates with a Mori fibre space, by Lemma \ref{l-LMMP}(1).
But if $P \equiv 0$, then $K_X$ is not pseudo-effective as $B$ is big, and we can run an LMMP on $K_{X}$
which terminates with a Mori fibre space [\ref{BCHM}]. Note that in both cases the LMMP preserves the
lc property of $(X,B)$ and the $\Q$-factorial klt property of $(X,0)$: in the first case 
the klt property of $(X,0)$ is preserved since the LMMP is also an LMMP on $K_X+\tilde{B}$ for some 
klt $(X,\tilde{B})$; in the second case
the lc property of $(X,B)$ is preserved as $K_X+B\equiv 0$. Also note that in either case the LMMP does not contract any $B_k$
because $B_k$ is big (although some of its components may be contracted).
So in either case replacing $X$ with the Mori fibre space obtained we may assume that
we already have a $K_X$-negative Mori fibre structure $X\to T$.

Let  $F$ be a general fibre of $X\to T$.
Since $B_k$ is big, $B_k|_{F}$ is big too.
Restricting to $F$ and applying induction on dimension we can reduce the problem
to the case $\dim T=0$, that is, when $X$ is a $\Q$-factorial klt Fano variety of Picard number one.
Pick a small number $\epsilon>0$. For each $j,k$ take a rational number $b_{j,k}'\le b_{j,k}$ such that
$b_{j,k}' \ge b - \epsilon$. Let $B'=\sum_k\sum_j b_{j,k}'B_{j,k}$. Then there is $P'\ge 0$
such that $K_X+B'+P'\equiv 0$ and $(X,B'+P')$ is lc. Now by [\ref{Kollar+}, Corollary 18.24],
$$
r(b-\epsilon)\le \sum_k\sum_j b_{j,k}'\le d + 1 .
$$
Therefore taking the limit when $\epsilon$
approaches $0$ we get $rb\le d + 1$ hence $r\le (d+1)/b$.\\
\end{proof}

Next we prove a result similar to Proposition \ref{p-toric},
though not as sharp, for the nef part of generalized polarized pairs.

\begin{prop}\label{p-bnd-comps}
Let $d$ be a natural number and $b$ a positive real number.
Assume that the ACC for generalized lc thresholds (Theorem \ref{t-acc-glct}) holds
in dimension $d$. Then there is a natural
number $p$ depending only on $d,b$ such that if
$(X',B'+M')$ is a generalized lc polarized pair of dimension $d$ with data
$X\overset{f}\to X'\to Z$ and $M$ satisfying
\begin{itemize}
\item[(i)]
$Z$ is a point,
\item[(ii)]
$M=\sum_1^r \mu_j M_j$ where $M_j$ are nef Cartier divisors  and $\mu_j\ge b$,
\item[(iii)]
$M_j'$ is a big $\Q$-Cartier divisor for every $j$, and
\item[(iv)]
$K_{X'}+B'+M'+P'\equiv 0$ for some pseudo-effective $\R$-Cartier divisor $P'$,\\
\end{itemize}
then the number of the $\mu_j$ is at most $p$, that is, $r\le p$.
\end{prop}

Before giving the proof we prove a related local statement.

\begin{prop}\label{p-bnd-comps-local}
Let $d$ be a natural number and $b$ a positive real number.
Assume that Theorem \ref{t-acc-glct} and Proposition \ref{p-bnd-comps} hold
in dimension $<d$. Then there is a natural
number $q$ depending only on $d,b$ such that if
$(X',B'+M')$ is a $\Q$-factorial generalized lc polarized pair of dimension $d$ with data
$X\overset{f}\to X'\to Z$ and $M$, and if 
\begin{itemize}
\item[(i)]
 $x'\in X'$ is a (not necessarily closed) point,
\item[(ii)]
$M=\sum_1^r \mu_j M_j$ where $M_j$ are nef$/Z$ Cartier divisors and $\mu_j\ge b$,
\item[(iii)]
$M_j$ is not relatively numerically zero over any neighborhood of $x'$, for every $j$, and
\item[(iv)]
$(X',0)$ is klt,\\
\end{itemize}
then the number of the $\mu_j$ is at most $q$, that is, $r\le q$.
\end{prop}
\begin{proof}
\emph{Step 1}.
Let $C'$ be the closure of $x'$. By $({\rm iii})$, the codimension of $C'$ in $X'$ is at least two.
By adding appropriate divisors to $B'$ and shrinking $X'$ we can assume $C'$ is a
generalized lc centre of $(X',B'+M')$: to be more precise, let $W$ be the
blowup of $X'$ along $C'$; we can assume $X\to X'$ factors through $W$;
now take a general sufficiently ample divisor on $W$ and let $A$ be its pullback to $X$; if
$\alpha$ is the generalized lc threshold of $A'$ near $x'$ with respect to $(X',B'+M')$, then
$(X',B'+\alpha A'+M')$ is generalized lc near $x'$ with boundary part $B'+\alpha A'$
and nef part $M$, and $C'$ is a
generalized lc centre of $(X',B'+\alpha A'+M')$; the point is that after shrinking $X'$ we can assume
$f^*A'=A+E$ where $E\neq 0$ is effective with large coefficients,
and that every component of $E$ maps onto $C'$ so adding $\alpha A'$ creates deeper
singularities only along $C'$. Now we may replace $B'$ with $B'+\alpha A'$.\\

\emph{Step 2}.
By Step 1, we can assume that there is a
prime divisor $S$ on $X$ mapping onto $C'$ whose generalized log discrepancy with respect to $(X',B'+M')$ is
$0$. Since $(X',0)$ is $\Q$-factorial klt, by Lemma \ref{l-extract-div-2}, there is an extremal birational contraction
$\phi : X''\to X'$ which extracts  $S''$, the birational transform of $S$, and $X''$ is $\Q$-factorial. We can write
$$
K_{X''} + B'' + M''=\phi^*(K_{X'} + B' + M')
$$
where $B''$ is the sum of $S''$ and the birational transform of $B'$, and $M''$ is
the pushdown of $M$. Writing
$$
K_X + B + M = f^*(K_{X'} + B' + M')
$$
we can see that $B''$ is just the pushdown of $B$.

We claim that $M_j''$ is not numerically trivial over any neighborhood of $x'$ for any $j$
which in turn implies that $M_j''$ is ample$/X'$. If this is not
true for some $j$, then we can write $f^*M_j'=M_j+\tilde{E}_j$ where $\tilde{E}_j\ge 0$ and $S$ is not a
component of $\tilde{E}_j$. But then for any general closed point $y'\in C'$, the fibre $f^{-1}\{y'\}$
is not inside $\Supp \tilde{E}_j$, so the fibre does not intersect $\Supp \tilde{E}_j$, 
by [\ref{KM}, Lemma 3.39(2)]. Therefore, $\tilde{E}_j=0$ 
over the generic point of $C'$, that is over $x'$, hence $M_j$ is numerically trivial over some neighborhood 
of $x'$, a contradiction.\\ 

\emph{Step 3.}
We can assume the induced map $g\colon X\bir X''$ is a morphism.
To ease notation we replace $S''$ with its normalization and denote the induced morphism
$S\to S''$ by $h$. By generalized adjunction and usual adjunction, we can write
$$
K_{S''} + B_{S''} + M_{S''} =(K_{X''} + B'' + M'') |_{S''}\equiv 0/C'
$$
and 
$$
K_{S''}+\Delta_{S''}=(K_{X''}+B'')|_{S''}.
$$
Write $g^*M_j'' = M_j + E_j$ where  $E_j\ge 0$ is exceptional$/X''$.
Then
$$
M_j'' |_{S''} = h_*({M_j}|_S) + h_*({E_j}|_S)
$$
and
$$
M'' |_{S''}=\sum \mu_j h_*({M_j}|_S)+\sum \mu_jh_*({E_j}|_S)=M_{S''}+\sum \mu_jh_*({E_j}|_S)
$$
and
$$
B_{S''}=\Delta_{S''}+ \sum \mu_jh_*({E_j}|_S) .
$$

Let $V$ be a prime divisor on $S''$ and $b_V$ be its coefficient in $B_{S''}$.
Then, by the proof of Proposition \ref{p-adj-dcc},
$$
b_V\ge 1 - \frac{1}{l}+\sum \frac{\mu_jn_j}{l}
$$
for some natural number $l$ and integers $n_j\ge 0$. Moreover, $n_j>0$ if $V$ is a component of
 $h_*({E_j} |_S)$. This in particular shows that there is a natural number $s$ depending only on $b$
such that $V$ is a component of $h_*({E_j}|_S)$ for at most $s$ of the $j$ because $\sum \mu_jn_j\le 1$.\\

\emph{Step 4}.
Let $F''$ be a general fibre
of the induced map $S'' \to C'$ and $F$ the corresponding fibre of $S\to C'$.
Restricting to $F''$ as in Remark \ref{r-g-sing}(6), we get
$$
K_{F''} + B_{F''} + M_{F''} =(K_{S''} + B_{S''} + M_{S''})|_{F''} \equiv 0 .
$$
Also we get  
$$
K_{F''}+\Delta_{F''}:=(K_{S''}+\Delta_{S''})|_{F''}.
$$
Denote the morphism $F\to F''$ by $e$. Since $F''$ is a general fibre, restricting Weil divisors 
on $S''$ to $F''$ makes sense, and if $P$ is a Weil divisor on $S$, 
then we have $(h_*P)|_{F''}=e_*(P|_F)$. 
Therefore,
$$
M_j''|_{F''}=e_*(E_j |_F)+ e_*(M_j|_F), ~~~M_{F''}=\sum \mu_je_*(M_j|_F),
$$
and
$$
B_{F''}=B_{S''}|_{F''}=(\Delta_{S''}+\sum \mu_jh_*({E_j}|_S))|_{F''}=\Delta_{F''}+\sum  \mu_j e_*(E_j |_F).
$$

Since $F''$ may not be $\Q$-factorial, we need to make some further constructions. 
Let $(H'',\Delta_{H''})$ be a $\Q$-factorial
dlt model of $(F'',\Delta_{F''})$ and $\psi\colon H''\to F''$ the corresponding morphism.
By definition
$$
K_{H''}+\Delta_{H''}=\psi^*(K_{F''}+\Delta_{F''})
$$
and the exceptional divisors of $\psi$ all appear with coefficient $1$ in $\Delta_{H''}$.
Moreover, we can write
$$
K_{H''}+B_{H''}+M_{H''}=\psi^*(K_{F''}+B_{F''}+M_{F''})\equiv 0
$$
where $B_{H''}$ is the sum of the birational transform of $B_{F''}$ and the reduced
exceptional divisor of $\psi$. 

We can assume $c\colon F\bir H''$ is a morphism. By construction,
$$
\psi^*(M_j''|_{F''})=c_*(E_j |_F)+ c_*(M_j|_F) 
$$
which is big, and 
$$
M_{H''}=\sum \mu_jc_*(M_j|_F)~~\mbox{and}~~B_{H''}=\Delta_{H''}+\sum \mu_jc_*(E_j |_F).
$$ 
Moreover, since the exceptional divisors of $\psi$ are components of $\rddown{\Delta_{H''}}$, the divisor 
$\sum \mu_jc_*(E_j|_F)$ has no exceptional component, so it is just the birational transform of $\sum \mu_je_*(E_j|_F)$.\\

\emph{Step 5.}
Run an LMMP on $K_{H''}$. It terminates with a Mori fibre space $\overline{H}'' \to \overline{T}''$
and the generalized lc
property of $({H''}, B_{H''} + M_{H''})$ is preserved by the LMMP. Since
$c_*(E_j |_F)+c_*(M_j |_F)$ is big,  its pushdown to $\overline{H}''$ is also big, hence
ample over $\overline{T}''$. Let ${\overline G}''$ be a general fibre of the above Mori fibre space.
Then restriction to ${\overline G}''$ gives
$$
K_{{\overline G}''} + B_{{\overline G}''} + M_{{\overline G}''}=
(K_{{\overline H}''} + B_{{\overline H}''} + M_{{\overline H}''})|_{{\overline G}''} \equiv 0 .
$$
By construction, $M_{{\overline G}''}=\sum \mu_ja_*(M_j|_F)|_{{\overline G}''}$ where we can assume
$a\colon F\bir \overline{H}''$ is a morphism.
Applying Proposition \ref{p-bnd-comps} and rearranging the indexes, we can assume that there
is a natural number $t$ depending only on $d,b$ such that $a_*(M_j|_F)|_{{\overline G}''}\equiv 0$
for every $j>t$. But then $a_*(E_j |_F)|_{{\overline G}''}$ is big for each $j>t$.

For each $j>t$ choose a component $W_j$ of $a_*(E_j |_F)$ which is ample over $\overline{T}''$.
By construction, $W_j$ is the birational transform of a component $U_j$ of $e_*(E_j |_F)=(h_*(E_j|_S))|_{F''}$
and $U_j$ in turn is a component of $V_j\cap F''$ for some component $V_j$ of $h_*(E_j|_S)$. 
Moreover, $W_k=W_j$ if and only if $U_k=U_j$ if and only if $V_k=V_j$. 
By Step 3, for each $k$, $V_k=V_j$ for at most $s$ of the $j$. 
Thus for each $k$, $W_k=W_j$ for at most $s$ of the $j$.    
On the other hand, by Steps 3 and 4,  the $V_j$ appear as components  
of $B_{{S}''}$ with coefficient $\ge \min\{b,\frac{1}{2}\}$, and there exist 
at least $\frac{r-t}{s}$ such components. Similarly the $W_j$ appear as components  
of $B_{{\overline H}''}$ with coefficient $\ge \min\{b,\frac{1}{2}\}$, and there exist 
at least $\frac{r-t}{s}$ such components. 
Now apply Proposition \ref{p-toric} to $({\overline G}'', B_{{\overline G}''})$
to deduce that $\frac{r-t}{s}$ is bounded hence $r$ is bounded by
some $q$.\\
\end{proof}

\begin{proof}(of Proposition \ref{p-bnd-comps})
We argue by induction on the dimension $d$. The case $d = 1$ is clear.
Suppose that the proposition holds in dimension $< d$.

\emph{Step 1.}
Since $(X',B'+M')$ is generalized lc and $K_{X'}+B'$ is $\R$-Cartier, $(X',B')$ is lc.
Let $(X'',B'')$ be a $\Q$-factorial dlt model of $(X',B')$ and $\phi\colon X''\to X'$ the corresponding morphism.
We may assume $X\bir X''$ is a morphism. For each $j$, we have $\phi^*M_j'=M_j''+E_j''$ where $E_j''\ge 0$ 
is exceptional$/X'$ and $M_j''$ is the pushdown of $M_j$.
So
$$
K_{X''}+B''+\sum \mu_jE_j''+M''=\phi^*(K_{X'}+B'+M')
$$
where $M''$ is the pushdown of $M$.
Since the exceptional divisors of 
$\phi$ are components of $\rddown{B''}$ and since $(X',B'+M')$ is generalized lc, 
we deduce $E_j''=0$ for every $j$, hence $M_j''=\phi^*M_j'$ for every $j$ and $M''=\phi^*M'$.
Thus we may replace
$X'$ with $X''$, hence assume that $(X',B')$ is $\Q$-factorial dlt.\\

\emph{Step 2.}
If $P' \not\equiv 0$, then $K_{X'}+B'+M'$ is not pseudo-effective and so
we can run an LMMP on $K_{X'}+B' + M'$ which terminates with a Mori
fibre space, by Lemma \ref{l-LMMP}(1).
But if $P' \equiv 0$, then $K_{X'}+B'$ is not pseudo-effective as $M'$ is big and so we can run an LMMP on $K_{X'}+B'$
which terminates with a Mori fibre space. Note that in both cases the generalized lc property of $(X', B'+ M')$
is preserved: in the second case we use the fact $K_{X'} + B' + M' \equiv 0$. Also note that in both cases
none of the $M_j'$ is contracted by the LMMP since $M_j'$ is big.
In either case we can replace $X'$ with the Mori fibre space hence we may assume
we already have a Mori fibre structure $X'\to T'$. Let $F'$ be a general fibre of this fibre
space. Since $M_j'$ is big, $M_j'|_{F'}$ is big too.
Restricting to $F'$ and applying induction on dimension we can reduce the problem
to the case $\dim T'=0$, that is, when $X'$ is a Fano variety of Picard number one.
\\

\emph{Step 3.}
Perhaps after changing the indexes we may write $M_j'\equiv \lambda_{j} M_1'$ such that
$\lambda_j\ge 1$ for every $j$. Now we define $\tilde{\mu}_j$ as follows:
initially let $\tilde{\mu}_j={\mu}_j$; next decrease $\tilde\mu_2$ and instead
increase $\tilde\mu_1$ as much as possible so that
$$
(X',B'+\sum_{j\neq 2} \tilde\mu_jM_j')
$$
is generalized lc and
$$
K_{X'}+B'+\sum \tilde\mu_jM_j'+P'\equiv 0 .
$$
Either we hit a generalized lc threshold, i.e.
$(X',B'+\sum_{j\neq 2} \tilde\mu_jM_j')$ is generalized lc but not generalized klt, or that we reach $\tilde\mu_2=0$.
If the first case happens, we stop. But if the second case happens we repeat
the process by decreasing $\tilde\mu_3$ and increasing $\tilde \mu_1$, and so on.

We show that the above process involves only a bounded number of the $\mu_j$. Let
$l$ be the smallest number such that $\tilde\mu_j=\mu_j$ for every $j> l$.
We want to show that $l$ is bounded depending only on $d,b$. We can assume $l>1$.
By construction,
$$
\tilde\mu_1\ge \sum_{j\le l-1} \mu_j\lambda_j \ge \sum_{j\le l-1} \mu_j\ge (l-1)b
$$
so it is enough to show that $\tilde\mu_1$ is bounded depending only on $d,b$.
If $M_1$ is not numerically trivial over $X'$, then
the generalized lc threshold of $M_1'$ with respect to $(X', B')$
is finite and bounded from above by Theorem \ref{t-acc-glct}, and this in turn implies
boundedness of $\tilde\mu_1$.
But if $M_1$ is numerically trivial over $X'$, then again $\tilde\mu_1$ is bounded from above but for a
different reason: by the cone theorem $X'$ can be covered by curves $\Gamma'$
such that $-(K_{X'}+B')\cdot \Gamma'\le 2d$ which in turn implies that
$\tilde\mu_1 M_1'\cdot \Gamma'\le 2d$ hence $\tilde\mu_1 M_1\cdot \Gamma\le 2d$ where $\Gamma \subset X$ is the
birational transform of $\Gamma'$. This is possible only if $\tilde\mu_1$ is bounded
from above since $M_1$ is big and Cartier and hence $M_1\cdot \Gamma\ge 1$.

If at the end of the process $\tilde \mu_j=0$ for every $j\ge 2$, then the above arguments show that
$r$ is indeed bounded by some number $p$. But if $\tilde \mu_j>0$ for some $j\ge 2$, then we replace $M$ with
$\sum_{j\neq l} \tilde\mu_jM_j$ and replace $P'$ with $P'+\tilde{\mu}_lM_l'$ 
where $l$ is as above, and rearrange the indexes. We can then assume
that $(X',B'+{M}')$ is generalized lc but not generalized klt.\\

\emph{Step 4.}
The arguments of Step 3 show that, after replacing $X$,  we can assume that there is a
prime divisor $S$ on $X$ exceptional over $X'$ whose generalized log discrepancy with respect to $(X',B'+M')$ is
$0$. Since $(X',0)$ is $\Q$-factorial klt, by Lemma \ref{l-extract-div-2}, there is an extremal contraction
$\phi : X''\to X'$ which extracts  $S''$, the birational transform of $S$. We can write
$$
K_{X''} + B'' + M''=\phi^*(K_{X'} + B' + M')
$$
where $B''$ is the sum of $S''$ and the birational transform of $B'$ and $M''$ is
the pushdown of $M$.

Since $\rho(X') = 1$ and $\phi$ is extremal, $\rho(X'') = 2$.
Moreover,
$$
K_{X''}+B''+M''+{P}''\equiv 0
$$
where ${P}''$ is the pullback of $P'$ on $X'$.
Since $\rho(X') = 1$, $P'$ and so $P''$ is semi-ample, hence
we may assume that $(X'', B'' + P'' + M'')$ is generalized lc with
boundary part $B'' + P''$ and nef part $M$.

Since $S''$ is a component of $\rddown{B''}$, $(X'', B'' -\delta S''+ P'' + M'')$ is generalized lc 
where $\delta>0$ is small, and
$$
-\delta S''\equiv K_{X''}+B''-\delta S''+P''+M''.
$$
So by Lemma \ref{l-LMMP}(1), we can run an
LMMP on $-S''$ which ends up with a Mori fibre space $X'''\to T''' .$
Note that by construction $X'''$ has Picard number one or two: in any case
one of the extremal rays of $X'''$
corresponds to the Fano contraction $X'''\to T'''$ and $S'''$ is positive on this ray.

We may assume that both $g\colon X\bir X''$ and $h\colon X\bir X'''$
are morphisms.\\

\emph{Step 5.}
Consider the case $\dim T'''>0$. Then the Picard number
$\rho(X''') = 2$, hence $X'' \dashrightarrow X'''$ is an isomorphism in codimension one.
Moreover, by restricting to the general fibres
of $X'''\to T'''$ and applying induction we may assume $M_j'''\equiv 0/T'''$ for all but a
bounded number of $j$. For any such $j$, $M_j'''$ is not big, hence $M_j''$ is not big too. Thus $M_j''$ is ample$/X'$
otherwise $M_j''$ would be the pullback of $M_j'$ which is big, a contradiction.
Let $ C':= \phi(S'')$ and let $x'$ be the generic point of $C'$.
Then $M_j'''\equiv 0/T'''$ implies that $M_j$ is not numerically trivial over any neighborhood of $x'$. Now apply
Proposition \ref{p-bnd-comps-local} to $(X',B'+P'+M')$ at $x'$ to bound the number of such $j$. Therefore
$r$ is indeed bounded by some number $p$ depending only on $d,b$.\\

\emph{Step 6.}
We can now assume $\dim T'''=0$. Let $\tilde{X}''\to X'''$ be the last step
of the LMMP which contracts some divisor $\tilde R''$. Let $x'''$ be the generic point
of the image of $\tilde R''$. For each $j$, either $M_j''$ is ample over $X'$
or $\tilde{M}_j''$ is ample over $X'''$ where $\tilde{M}_j''$ is the pushdown of $M_j$ via
$X\bir \tilde X''$ which we can assume to be a morphism. So either $M_j$ is not numerically
trivial over any neighborhood of $x'$ or that it is not numerically
trivial over any neighborhood of $x'''$. Now apply
Proposition \ref{p-bnd-comps-local} to $(X',B'+P'+M')$ and $(X''',B'''+P'''+M''')$
at $x'$ and $x'''$ to bound $r$ by some number $p$ depending only on $d,b$.\\
\end{proof}


\section{ACC for generalized lc thresholds}

In this section, we reduce the ACC for generalized lc thresholds (Theorem \ref{t-acc-glct}) 
to the Global ACC (Theorem \ref{t-global-acc}) in lower dimension by adapting a standard 
argument due to Shokurov. We create an appropriate generalized lc centre of codimension one 
and restrict to it to do induction.

\begin{prop}\label{p-global-acc-to-glct}
Assume that Theorem \ref{t-global-acc} holds in dimension $\le d-1$. Then Theorem \ref{t-acc-glct} holds in dimension $d$.
\end{prop}
\begin{proof}
Applying induction we may assume that Theorem \ref{t-acc-glct} holds in dimension $\le d-1$.
If Theorem \ref{t-acc-glct} does not hold in dimension $d$, then there exist a sequence of generalized
lc polarized pairs
$(X_i',B_{i}'+M_{i}')$ of dimension $d$ with data $X_i \overset{f_i}\to X_i'\to Z_i$ and $M_i=\sum \mu_{j,i}M_{j,i}$,
and divisors $D_{i}'$ and $N_i=\sum \nu_{k,i}N_{k,i}$
satisfying the assumptions of
the theorem but such that the generalized lc thresholds $t_i$ of $D_{i}'+N_{i}'$ with
respect to $(X_i',B_{i}'+M_i')$ form a strictly increasing sequence of numbers.
We may assume that $0 < t_i < \infty$ for every $i$. Since the problem is local over $X_i'$, we can
assume $X_i'\to Z_i$ is the identity morphism. Moreover, we can discard any $\mu_{j,i}$ and $\nu_{k,i}$ if they are zero.

By definition,
$$
(X_i',B_{i}'+t_iD_{i}'+M_i'+t_iN_{i}')
$$
is generalized lc with boundary part $B_{i}'+t_iD_{i}'$ and nef part $M_i+t_iN_{i}$ but
$$
(X_i',B_{i}'+a_iD_{i}'+M_i'+a_iN_{i}')
$$
is not generalized lc for any $a_i>t_i$.

If $\rddown{B_i'}\neq \rddown{B_i'+t_iD_i'}$ for infinitely many $i$, then we can easily get a contradiction
as the $t_i$ can be calculated in terms of the coefficients of $B_i'$ and $D_i'$.
Thus we may assume that $\rddown{B_i'}= \rddown{B_i'+t_iD_i'}$ for every $i$. In particular, this means that
there is a generalized lc centre of
$$
(X_i',B_{i}'+t_iD_{i}'+M_i'+t_iN_{i}')
$$
of codimension $\ge 2$ which is not a generalized lc centre of
$(X_i',B_{i}'+M_i')$.

We may assume that the given morphism $f_i\colon X_i\to X_i'$ is a log resolution of
$(X_i',B_{i}'+t_iD_{i}')$. Let $\Delta_i':=B_i'+t_iD_i'$ and let $R_i:=M_i+t_iN_i$. We can write
$$
K_{X_i}+\Delta_i+R_i=f_i^*(K_{X_i'}+\Delta_i'+R_i')+E_i
$$
where $\Delta_i$ is the sum of the birational transform of $\Delta_{i}'$ and the reduced
exceptional divisor of $f_i$, and $E_i\ge 0$ is
exceptional$/X_i'$. Then the pair $(X_i,\Delta_i)$ is
lc but not klt; more precisely there is a component of $\rddown{\Delta_i}$
which is not a component of $E_i$; moreover, there is such a component which is
exceptional$/X_i'$ by the last paragraph.
In addition, the set of the coefficients of all the
$\Delta_i$ union with  $\{\mu_{j,i}, t_i\nu_{k,i}\}$ satisfies the DCC by Lemma \ref{l-ACC-DCC}.

Run an LMMP$/X_i'$ on $K_{X_i}+\Delta_i+R_i$ with scaling of some ample divisor which 
is also an LMMP$/X_i'$ on $E_i$. Since $E_i$ is effective and exceptional$/X_i'$,
the LMMP ends on a model $X_i''$ on which $E_i''=0$ (as in the proof of Lemma \ref{l-extract-divs-1}).
In particular,
$$
K_{X_i''}+\Delta_i''+R_i''\equiv 0/X_i' .
$$

Let $S_i$ be a component of $\rddown{\Delta_i}$ exceptional$/X_i'$ but not a component of $E_i$.
Since the LMMP only contracts components of $E_i$, this $S_i$ is not contracted$/X_i''$.
Define $\Delta_{S_i''}$ by the generalized adjunction
$$
K_{S_i''}+\Delta_{S_i''}+R_{S_i''}=(K_{X_i''}+\Delta_i''+R_i'')|_{S_i''} .
$$
Then the set of the coefficients of all the $\Delta_{S_i''}$ satisfies DCC by Proposition \ref{p-adj-dcc}.
By construction
$$
K_{S_i''}+\Delta_{S_i''}+R_{S_i''}\equiv 0/X_i' .
$$
Let $S_i''\to V_i'$ be the contraction given by the Stein factorization of
 $S_i''\to X_i'$ and let $F_i''$ be a general fibre of $S_i''\to V_i'$.
We can write
$$
K_{F_i''}+\Delta_{F_i''}+R_{F_i''}=(K_{S_i''}+\Delta_{S_i''}+R_{S_i''})|_{F_i''}\equiv 0
$$
as in Remark \ref{r-g-sing} (6): here $\Delta_{F_i''}=\Delta_{S_i''}|_{F_i''}$
and $R_{F_i''}=R_{S_i''}|_{F_i''}$ is the pushdown of $R_i|_{F_i}=(M_i+t_iN_i)|_{F_i}$ where $F_i$
is the fibre of $S_i\to V_i'$ corresponding to $F_i''$.

Suppose that we can choose the $S_i$ such that

\par \vskip 1pc
$(*)~~$ the set of the coefficients of all the $\Delta_{F_i''}$ together with
$\{\mu_{j,i}\mid M_{j,i}|_{F_i}\not\equiv 0\}\cup \{t_i\nu_{k,i}\mid N_{k,i}|_{F_i}\not\equiv 0\}$
 does not satisfy ACC.

\par \vskip 1pc
But then $(*)$ contradicts Theorem \ref{t-global-acc}. So it is enough to find the $S_i$ so that
 $(*)$ holds. We will show that $(*)$ holds if for each $i$ we can find $S_i$ satisfying:

\par \vskip 1pc
$(**)~~$  $(D_i''+N_i'')|_{F_i''}$ is not numerically trivial

\par \vskip 1pc
{\flushleft where} $D_i$ on $X_i$ is the birational transform of $D_i'$ and $D_i''$ is the pushdown of $D_i$; 
here we can assume $g_i \colon X_i \bir X_i''$ is a morphism.
Indeed, let $B_i$ be the sum of the birational transform of $B_i'$ plus the reduced exceptional
divisor of $f_i$, and $B_i''$ its pushdown on $X_i''$. By generalized adjunction we can write
$$
K_{S_i''}+B_{S_i''}+M_{S_i''}=(K_{X_i''}+B_i''+M_i'')|_{S_i''} .
$$
Write $g_i^*(N_i'') =N_i + Q_i$. Then $N_i''|_{F_i''}=N_{F_i''}+Q_{F_i''}$ where
$N_{F_i''}$ is the pushdown of $N_i|_{F_i}$ and $Q_{F_i''}$ is the pushdown of
$Q_i|_{F_i}$. If  $N_i|_{F_i}\not\equiv 0$ for every $i$, then $(*)$ is satisfied.
So we can assume $N_i|_{F_i}\equiv 0$ for every $i$, hence by $(**)$ we have
$$
(D_i''+N_i'')|_{F_i''}\equiv D_{F_i''}+Q_{F_i''}\neq  0
$$
for every $i$ where $D_{F_i''}:=D_i''|_{F_i''}$.
But now $\Delta_{i}''=B_{i}''+t_iD_{i}''$ and 
$\Delta_{S_i''}=B_{S_i''}+t_i(D_{S_i''}+Q_{S_i''})$ where $D_{S_i''}:=D_i''|_{S_i''}$
and $Q_{S_i''}$ is the pushdown of $Q_i|_{S_i}$. Moreover, since $D_{S_i''}+Q_{S_i''}\neq 0$ near $F_i''$,
Proposition \ref{p-adj-dcc} and its proof show that the set of the coefficients of all the
$\Delta_{S_i''}$ near $F_i''$ does not satisfy ACC.
Thus the set of the coefficients of all the $\Delta_{F_i''}$
does not satisfies ACC, hence  $(*)$ holds.

Finally we show that $(**)$ holds.
By the negativity lemma, we can write
$$
f_i^*(D_i'+N_i')=D_i+N_i+P_i
$$
where $P_i\ge 0$ is exceptional$/X_i'$. By the definition of $t_i$
and the assumption  $\rddown{B_i'}= \rddown{B_i'+t_iD_i'}$, there is a component of $P_i$
which is a component of $\rddown{\Delta_i}$ but not a component of $E_i$. In fact any component
of $P_i$ not contracted$/X_i''$, is of this kind.
Since $P_i''\neq 0$ is exceptional$/X_i'$,
by the negativity lemma [\ref{BCHM}, Lemma 3.6.2],
there is a component $S_i''$ of $P_i''$ with a covering family of curves $C$ (contracted over $X_i'$) 
such that $P_i''\cdot C<0$.
So $(D_i''+N_i'')\cdot C>0$ for such curves $C$, hence $(D_i''+N_i'')|_{S_i''}$ is not numerically trivial 
over general points of $V_i'$ 
 which implies that we can choose the $S_i$ so that $(**)$ holds.\\
\end{proof}


\section{Global ACC}\label{s-global-acc}

In this section, we show that Global ACC (Theorem \ref{t-global-acc}) in dimension $< d$
and ACC for generalized lc thresholds (Theorem \ref{t-acc-glct} ) in dimension $d$
together imply Global ACC in dimension $d$.
We first deal with the pairs which are generalized lc but not generalized klt.
For the general case, we will use Proposition \ref{p-bir-nM-n-large} and
do induction on the number of summands in the nef part of the pair,
as illustrated in the introduction.
The starting point of the induction is the important result [\ref{HMX2}, Theorem 1.5] which 
proves the statement when the nef part is zero.

\begin{prop}\label{p-global-acc-glc}
Assume Theorem \ref{t-global-acc} holds in dimension $\le d-1$.
Then Theorem \ref{t-global-acc} holds in dimension $d$ for those $(X',B'+M')$ which are not generalized klt.
\end{prop}
\begin{proof}
\emph{Step 1.}
Extending $\Lambda$ we can assume $1\in \Lambda$.
If the proposition does not hold, then there is a sequence of generalized lc but not klt polarized pairs
$(X_i',B_i'+M_i')$ with data $X_i \overset{f_i}\to X_i'\to Z_i$ and $M_i=\sum \mu_{j,i}M_{j,i}$
satisfying the assumptions of \ref{t-global-acc} but such that
the set of the coefficients of all the $B_i'$ together with the $\mu_{j,i}$ does not satisfy ACC.

We may assume that $f_i : X_i \to X_i'$ is a log resolution of $(X_i',B_i')$.
Let $B_i$ be the sum of the birational transform of $B_i'$ and the reduced exceptional
divisor of $f_i$. We can write
$$
K_{X_i}+B_i+M_i=f_i^*(K_{X_i'}+B_i'+M_i')+E_i
$$
where $E_i\ge 0$ is exceptional$/X_i'$. 
We can run an LMMP$/X_i'$ on $K_{X_i}+B_i+M_i$ with scaling of some ample divisor 
which contracts $E_i$ and terminates with some model (as in the proof of Lemma \ref{l-extract-divs-1}). 
Moreover, by the generalized non-klt assumption, we can choose $f_i$ so that 
there is a prime divisor $S_i$ on $X_i$ which is a component of $\rddown{B_i}$ but not 
a component of $E_i$, hence it is not contracted by the LMMP.
Replacing $X_i'$ with the model given by the LMMP allows us to assume that $(X_i',B_i')$ is $\Q$-factorial dlt 
and that we have a component $S_i'$ of $\rddown{B_i'}$.\\

\emph{Step 2.}
Write $B_i'=\sum b_{k,i}B_{k,i}'$ where $B_{k,i}'$ are the
distinct irreducible components of $B_i'$.
If the set of all the coefficients $b_{k,i}$ is not finite, then we may assume that the
$b_{1,i}$ form a strictly increasing sequence in which case we let $P_i':=B_{1,i}'$. On the other hand, if
the set of all the coefficients $b_{k,i}$ is finite, then the set of all the $\mu_{j,i}$
is not finite hence we could assume that the $\mu_{1,i}$ form a strictly increasing sequence in which
case we let $P_i':=M_{1,i}'$. 
In either case we can run an LMMP on 
$$
K_{X_i'}+B_i'+M_i'-\epsilon P_i'\equiv -\epsilon P_i'
$$
for some small $\epsilon>0$ which ends with a Mori fibre space, by Lemma \ref{l-LMMP}(1).
The  generalized lc (and non-klt) property of $(X_i',B_i'+M_i')$ is preserved by the LMMP 
because $K_{X_i'}+B_i'+M_i' \equiv 0$.\\

\emph{Step 3.}
We first consider the case when $S_i'$ is not contracted by the LMMP in Step 2, for infinitely many $i$. 
Replacing the sequence we can assume this holds for every $i$.
In this case, we replace $X_i'$ with the Mori fibre space constructed, hence we can assume we already have a 
Mori fibre structure $X_i'\to T_i'$ and that $P_i'$ is ample$/T_i'$. 
Let $F_i'$ be a general fibre of $X_i'\to T_i'$. Then we can write
$$
K_{F_i'}+B_{F_i'}+M_{F_i'}=(K_{X_i'}+B_i'+M_i')|_{F_i'}
$$
where $B_{F_i'}=B_i'|_{F_i'}$ and $M_{F_i'}=M_i'|_{F_i'}$. Moreover, since $P_i'|_{F_i'}$ is ample,
the set of the coefficients of all the $B_{F_i'}$ together with the set
$\{\mu_{j,i}\mid M_{j,i}|_{F_i}\not\equiv 0\}$
is not finite where $F_i$ is the fibre of $X_i\to T_i'$ corresponding to $F_i'$.
So applying induction we can assume $\dim T_i'=0$. In particular,  $P_i'|_{S_i'}$ is not numerically trivial.

Now assume the LMMP of Step 2 contracts $S_i'$ at some step, for infinitely many $i$. 
Replacing the sequence we can assume this holds for every $i$. Replacing $X_i'$ 
we can assume $S_i'$ is contracted by the first step of the LMMP, say $X_i'\to X_i''$. 
Then $P_i'$ is ample over $X_i''$, hence $P_i'|_{S_i'}$ is not numerically trivial.

From now on we assume that $P_i'|_{S_i'}$ is not numerically trivial.\\

\emph{Step 4.} 
Apply generalized adjunction to get
$$
K_{S_i'}+B_{S_i'}+M_{S_i'}=(K_{X_i'}+B_i'+M_i')|_{S_i'}\equiv 0 .
$$
By Proposition \ref{p-adj-dcc}, the coefficients of $B_{S_i'}$ belong to a DCC set depending only on 
$d$ and $\Lambda$.
Moreover, $M_{S_i'}$ is the pushdown of $M|_{S_i}=\sum \mu_{j,i}M_{j,i}|_{S_i}$.
Thus by induction the set of the coefficients of all the $B_{S_i'}$ together with the set
$\{\mu_{j,i}\mid M_{j,i}|_{S_i}\not\equiv 0\}$  is finite. But this contradicts Lemma \ref{l-adjunction-2}.\\
\end{proof}

\begin{prop}\label{p-glct-to-global-acc}
Assume that Theorem \ref{t-global-acc} holds in dimension $\le d-1$ and that Theorem
\ref{t-acc-glct} holds in dimension $d$. Then Theorem \ref{t-global-acc} holds in dimension $d$.
\end{prop}
\begin{proof}
\emph{Step 1}.
If the statement is not true, then there is a sequence of generalized lc polarized pairs
$(X_i',B_i'+M_i')$ with data  $X_i \overset{f_i}\to X_i'\to Z_i$ and $M_i=\sum \mu_{j,i}M_{j,i}$ 
satisfying the assumptions of \ref{t-global-acc} but such that
the set of the coefficients of all the $B_i'$ and all the $\mu_{j,i}$ put together
satisfies DCC but not ACC. Write $B_i'=\sum b_{k,i}B_{k,i}'$ where $B_{k,i}'$ are the
distinct irreducible components of $B_i'$.

As in Steps 1 and 2 of the proof of Proposition \ref{p-global-acc-glc}, we can reduce the problem to the
situation in which  $X_i'$ is a $\Q$-factorial klt variety with a Mori fibre space structure. 
Restricting to the general fibres of the fibration and applying induction we can in addition assume $X_i'$ 
is Fano of Picard number one.

For each $i$, let $\sigma(M_i)$ be the number of the $\mu_{j,i}$.
Then, by Propositions \ref{p-toric} and \ref{p-bnd-comps}, we can assume that
the number of the components of $B_i'$ plus $\sigma(M_i)$ is bounded. Thus
we can assume that the number of the components of $B_i'$ and $\sigma(M_i)$ are both independent of $i$.
We just write $\sigma$ instead of $\sigma(M_i)$.

We will do induction on the number $\sigma$.
By [\ref{HMX2}, Theorem 1.5], the proposition holds when $\sigma = 0$, i.e. when $M_i = 0$ for every $i$.
So we can assume $\sigma>0$. We may also assume that $\sigma$ is minimal with respect to all sequences as above, even
if $\Lambda$ is extended to a larger set.

Replacing
the sequence we may assume that the numbers $b_{k,i}$ and $\mu_{j,i}$ form a (not necessarily strict) increasing sequence for each $k$ and each $j$, because they all belong to the DCC set $\Lambda$. By definition, $b_{k,i}\le 1$.
We show that the $\mu_{j,i}$ are also bounded from above, i.e. $\lim_i \mu_{j,i}<+\infty$ for every $j$:
this follows from the same arguments as in Step 3 of the proof of Proposition \ref{p-bnd-comps}
by considering the generalized lc threshold of $M_{j,i}'$ with respect to $(X_i',B_i'+M_i'-\mu_{j,i}M_{j,i}')$
if $M_{j,i}\not\equiv 0/X_i'$ for infinitely many $i$, or by applying boundedness of the
length of extremal rays otherwise.\\

\emph{Step 2}.
By Proposition \ref{p-global-acc-glc}, we may assume that $(X_i',B_i'+M_i')$ is generalized klt for every $i$.
In particular,  $({X_i'},B_i'+(1+\epsilon_i)M_{i}')$
is generalized klt, and $K_{X_i'}+B_i'+(1+\epsilon_i)M_{i}'$ is ample for some small $\epsilon_i> 0$,
noting that the Picard number $\rho(X_i') = 1$.
We may assume that $f_i : X_i \to X_i'$ is a log resolution of $(X_i',B_i')$ and can write
$$
K_{X_i}+B_i+(1+\epsilon_i)M_i=f_i^*(K_{X_i'}+B_{i}'+(1+\epsilon_i)M_{i}')+E_i
$$
where $B_i$ is the sum of the birational transform of $B_{i}'$ and the reduced
exceptional divisor of $f_i$, and $E_i\ge 0$ is exceptional$/X_i'$. So $K_{X_i}+B_i+(1+\epsilon_i)M_i$
is big, and since the $\mu_{j,i}$ are bounded from above,
we deduce that $K_{X_i}+B_i+\sum n M_{j,i}$ is also big for some fixed natural number $n\gg 1$
independent of $i$.

Now by Proposition \ref{p-bir-nM-n-large}, there exists a natural number $m$, independent of $i$,
such that $|m(K_{X_i}+B_i+\sum nM_{j,i})|$ defines a birational map for every $i$. In particular,
$$
H^0(X_i, \rddown{m(K_{X_i}+B_i+\sum nM_{j,i})})\neq 0
$$
so
$$
\rddown{m(K_{X_i}+B_i+\sum nM_{j,i})}\sim m\overline{D}_i
$$
for some
integral divisor $m\overline{D}_i\ge 0$. The coefficients of $m\overline{D}_i$ belong to $\N$,
a DCC set. Now let $D_i$ be the $\R$-divisor so that $mD_i$ is the sum
of $m\overline{D}_i$ and the fractional part $\langle m(K_{X_i}+B_i+\sum nM_{j,i})\rangle$.
Since $K_{X_i}+\sum nM_{j,i}$
is Cartier,  $mD_i=m\overline{D}_i+\langle mB_i \rangle$.
On the other hand, since the coefficients of $B_i$ belong to the DCC set $\Lambda\cap [0,1]$, the coefficients of
$\langle mB_i \rangle$ belong to a
DCC set as well by Lemma \ref{l-ACC-DCC}.
Therefore, the coefficients of $mD_i$ and hence of $D_i$ belong to a DCC set,
depending only on $\Lambda$. By extending $\Lambda$ we can assume that the coefficients of $D_i$
belong to $\Lambda$.

By construction,
$$
0\le {D}_i\sim_\R {K_{X_i}+B_i+\sum nM_{j,i}}
$$
which in turn implies that
$$
0\le {D}_i'\sim_\R {K_{X_i'}+B_i'+\sum nM_{j,i}'}
$$
$$
={K_{X_i'}+B_i'+M_i'}+\sum (n-\mu_{j,i})M_{j,i}'\equiv \sum (n-\mu_{j,i})M_{j,i}' .
$$
Note that we can assume $n-\mu_{j,i}>0$ for every $j,i$.\\

\emph{Step 3}.
By Lemma \ref{l-ordering-divs}, replacing the sequence $X_i'$ and reordering the indexes $j$,
we may assume that $M_{j,i}'\equiv \lambda_{j,i}M_{1,i}'$ so that for each $j$ the numbers $\lambda_{j,i}$
form a decreasing sequence.
By Step 2, we get
$$
D_i'\equiv \sum (n-\mu_{j,i})M_{j,i}'\equiv \sum (n-\mu_{j,i})\lambda_{j,i}M_{1,i}'=:\rho_iM_{1,i}'
$$
where we have defined
$$
\rho_i:=\sum (n-\mu_{j,i})\lambda_{j,i} .
$$
For each $j$, the numbers $n-\mu_{j,i}$ and
$\lambda_{j,i}$ form decreasing sequences hence the $\rho_i$ also form a decreasing sequence by Lemma \ref{l-ACC-DCC}.

Now let $N_i:=\sum_{j\ge 2}\mu_{j,i}M_{j,i}$ and let $u_i$ be the generalized lc threshold of
$D_i'$ with respect to $(X_i',B_i'+N_i')$.
Since $D_i'\equiv \rho_iM_{1,i}'$ we get
$$
K_{X_i'}+B_i'+N_i'+u_iD_i'\equiv K_{X_i'}+B_i'+N_i'+u_i\rho_iM_{1,i}' .
$$

Assume that $u_i\rho_i\ge \mu_{1,i}$ for every $i$. Let $v_i\le u_i$ be the number so that
$$
K_{X_i'}+B_i'+N_i'+v_iD_i'\equiv K_{X_i'}+B_i'+M_i'\equiv 0
$$
that is, $v_i=\frac{\mu_{1,i}}{\rho_i}$. As the $\mu_{1,i}$ form an increasing sequence
and the $\rho_i$ form a decreasing sequence, the $v_i$ form an increasing sequence.
Moreover, if the $\mu_{1,i}$ form a strictly increasing sequence, then
the $v_i$ also form a strictly increasing sequence.
Thus the set of the coefficients of all
the $B_i'+v_iD_i'$ together with the $\{\mu_{j,i} \mid j\ge 2\}$ is
a DCC set but not ACC.
Now $({X_i'},B_i'+v_iD_i'+N_i')$ is generalized lc with boundary part $B_i'+v_iD_i'$
and nef part $N_i$, and
$\sigma(N_i)<\sigma$ which contradicts the minimality assumption on $\sigma$ in Step 1.
Therefore, from now on we may assume that $u_i\rho_i < \mu_{1,i}$ for every $i$.\\

\emph{Step 4}.
Fix $i$. Let $\Sigma_i$ be the set of those elements
$(\alpha,\beta)\in [0,\frac{\mu_{1,i}}{\rho_i}]\times[0,\mu_{1,i}]$ such that
$$
K_{X_i'}+B_i'+N_i'+\alpha D_i'+\beta M_{1,i}'\equiv K_{X_i'}+B_i'+M_i'
$$
which is equivalent to $\alpha\rho_i+\beta=\mu_{1,i}$. Note that $(0,\mu_{1,i})\in \Sigma_i$
hence $\Sigma_i\neq \emptyset$.
Now let
$$
s_i=\sup \{\alpha \mid (\alpha,\beta)\in \Sigma_i ~~\mbox, \,
~~({X_i'},B_i'+\alpha D_i'+N_i'+\beta M_{1,i}') ~~\mbox{is generalized lc}\}
$$
where the pair in the definition has boundary part $B_i'+\alpha D_i'$ and nef part $N_i+\beta M_{1,i}$.
Letting $t_i=\mu_{1,i}-s_i\rho_i$ we get $(s_i,t_i)\in \Sigma_i$.

We show that $s_i$ is actually a maximum hence in particular
$$
({X_i'},B_i'+s_i D_i'+N_i'+t_i M_{1,i}')
$$
{is generalized lc}. If not, then there is a sequence
$(\alpha^l,\beta^l)\in \Sigma_i$ such that the $\alpha^l$ form a strictly increasing sequence
approaching $s_i$ and the $\beta^l$ form a strictly decreasing sequence approaching
$t_i$. Since
$$
({X_i'},B_i'+\alpha^l D_i'+N_i'+t_i M_{1,i}')
$$
is generalized lc, the generalized lc threshold of $D_i'$ with
respect to $({X_i'},B_i'+ N_i'+t_i M_{1,i}')$
is at least $\lim \alpha^l = s_i$ by Theorem \ref{t-acc-glct}.
So
$$
({X_i'},B_i'+s_i D_i'+N_i'+t_i M_{1,i}')
$$
is also generalized lc. Hence $s_i$ is indeed a maximum.
Note that $s_i\le u_i$.\\

\emph{Step 5}.
Since the coefficients of $D_i'$ belong to $\Lambda$ and since $u_i$ is the generalized lc threshold
of $D_i'$ with respect to $(X_i',B_i'+N_i')$, $u_i$ is bounded from above by Theorem \ref{t-acc-glct}. Thus $s_i$
is also bounded from above.
So we may assume the $s_i$ and the $t_i$ each form an increasing or a decreasing sequence
hence $s=\lim s_i$ and $t=\lim t_i$ exist.
Since the $\mu_{1,i}$ form an increasing sequence and the $\rho_i$ form a decreasing sequence,
the $s_i$ or the $t_i$ form an
increasing sequence. We will show that in fact the $t_i$ form an
increasing sequence.

Assume otherwise, that is, assume the $t_i$ form a decreasing sequence. We can assume it is  
strictly decreasing. Then the $s_i$ form a strictly increasing sequence.
Since
$$
({X_i'},B_i'+s_i D_i'+N_i'+t M_{1,i}')
$$
is generalized lc, we may assume that
$$
({X_i'},B_i'+s D_i'+N_i'+tM_{1,i}')
$$
is generalized lc too, by Theorem \ref{t-acc-glct}.
Now we can find $\tilde{s}_i>s_i$ such that
$(\tilde{s}_i, t)\in \Sigma_i$, that is, $\tilde{s}_i\rho_i+t=\mu_{1,i}$. Since the $\mu_{1,i}$ 
form an increasing sequence and the $\rho_i$ form a decreasing sequence, 
the $\tilde{s}_i$ form an increasing sequence. Moreover, since 
$$
t<t_i\le \mu_{1,i}\le \lim \mu_{1,i}~~\mbox{and}~~s(\lim \rho_i)+t=\lim (s_i\rho_i+t_i)=\lim \mu_{1,i}
$$ 
we deduce $\lim \rho_i>0$. Thus as  
$$
\lim (\tilde{s}_i\rho_i+t)=\lim \mu_{1,i},
$$ 
we get $\lim \tilde{s}_i=\lim {s}_i=s$. 
In particular this means $s \ge \tilde{s}_i>s_i$, hence  
$$
({X_i'},B_i'+\tilde{s}_i D_i'+N_i'+tM_{1,i}')
$$
is generalized lc which contradicts the maximality assumption of $s_i$ in Step 4.

So we have proved that the $t_i$ form an increasing sequence. Now by definition $s_i$ is
the generalized lc threshold of $D_i'$ with respect to
$$
({X_i'},B_i'+N_i'+{t}_i M_{1,i}').
$$
So they form a decreasing sequence by Theorem \ref{t-acc-glct}.
\\

\emph{Step 6}.
The purpose of this step is to modify $B_i'$ so that we can assume $s=\lim s_i=0$.
Let $\tilde{t}_i$ be the number so that $s\rho_i+\tilde{t}_i=\mu_{1,i}$. As $s_i\ge s$, 
$\tilde{t}_i \ge t_i \ge 0$, hence $(s,\tilde{t}_i)\in \Sigma_i$.
Since the $\mu_{1,i}$ (resp. $\rho_i$)
form an increasing (resp. decreasing) sequence, the $\tilde{t}_i$ form an increasing sequence.
Moreover, 
$$
\lim \tilde{t}_i=\lim(\mu_{1,i}-s\rho_i)=\lim(\mu_{1,i}-s_i\rho_i)=\lim t_i=t
$$
which implies $\tilde{t}_i \le t$.

We claim that
$$
(*) \hskip 1pc ({X_i'},B_i'+s D_i'+N_i'+\tilde{t}_i M_{1,i}')
$$
is generalized lc.
Indeed, let $c_i$ be the generalized lc threshold of $M_{1,i}'$
with respect to $({X_i'}, B_i'+s D_i'+N_i')$.
Then $c_i \ge t_i$ and by Theorem \ref{t-acc-glct}, we may assume that the $c_i$ form
a decreasing sequence.
Thus
$$
 c_i \ge \lim c_i \ge \lim t_i = t\ge \tilde{t}_i
$$
and the claim follows.

Now we define the boundary $C_i:=B_i+s\tilde{D}_i'$ on $X_i$ where $B_i$, as in Step 2, is the sum of the
birational transform of $B_i'$ and the reduced exceptional divisor of $X_i\to X_i'$,
and $\tilde{D}_i'$ is the birational
transform of $D_i'$. Then $C_i'=B_i'+sD_i'$ and
$$
({X_i'},C_i'+N_i'+\tilde{t}_i M_{1,i}')
$$
is generalized lc by $(*)$, and
$$
K_{X_i'} + C_i' +N_i' + \tilde{t}_i M_{1,i}' \equiv 0 .
$$
Moreover, the set of the coefficients of all the $C_i'$ union the set $\{\mu_{j,i} \mid j\ge 2\}\cup \{\tilde{t}_i\}$
satisfies DCC but not ACC (note that if the $\mu_{1,i}$ form a strictly increasing sequence, then so do the $\tilde{t}_i$).

On the other hand, let $G_i:=D_i+s\tilde{D}_i'$ and let $r_i:=\frac{s_i-s}{1+s}$.
Then
$$
0\le {G}_i\sim_\R {K_{X_i}+C_i+\sum nM_{j,i}}
$$
and $G_i' = (1+s)D_i'$, and
$$
 K_{X_i'}+C_i'+r_iG_i'+N_i'+t_iM_{1,i}'
= K_{X_i'}+B_i'+ s_iD_i'+N_i'+t_iM_{1,i}'
\equiv 0.
$$
The equality also shows
$$
(X_i', C_i'+r_iG_i'+N_i'+t_iM_{1,i}')
$$
is generalized lc and that $r_i$ is the generalized lc threshold of $G_i'$ with respect to
$({X_i'},C_i'+N_i'+{t}_i M_{1,i}')$.
Therefore extending $\Lambda$, replacing $B_i$ with $C_i$, replacing $\mu_{1,i}$ with $\tilde{t}_i$,
replacing $D_i$ with $G_i$,
and replacing $s_i$ with $r_i$ allow us to assume that $s=\lim s_i=0$.
\\

\emph{Step 7}.
After replacing $X_i$ we may assume that there is a prime divisor $S_i$ on $X_i$ whose
generalized log discrepancy with respect to the generalized lc polarized pair
$$
({X_i'}, B_i'+s_iD_i'+N_i'+t_iM_{1,i}')
$$
is $0$: this follows from our choice of $s_i,t_i$.

First assume that $S_i$ is not
contracted over $X_i'$ for every $i$ which means that $S_i'$ is a component of $\rddown{B_i'+s_iD_i'}$.
Let $d_i$ be the coefficient of $S_i'$ in $D_i'$ and let $p_i$ be the real number such that
$$
K_{X_i'}+B_i'+s_id_iS_i'+N_i'+p_iM_{1,i}'\equiv 0 .
$$
Obviously $p_i\le \mu_{1,i}$, and equality holds if and only if $s_i d_i S' \equiv 0$, i.e., $s_i d_i = 0$.
Since $s_id_iS_i'\le s_iD_i'$ and
$$
K_{X_i'}+B_i'+s_i{D}_i'+N_i'+t_iM_{1,i}'\equiv 0
$$
we have $t_i\le p_i$. Then from $\lim s_i=0$ and $\mu_1:=\lim \mu_{1,i}=\lim t_i$
we arrive at $\lim p_i=\mu_{1}$. So we may assume that the $p_i$ form an increasing
sequence approaching $\mu_{1}$.

Let $w_i$ be the generalized lc threshold of $M_{1,i}'$ with respect to
$$
(X_i', B_i'+s_id_iS_i'+N_i').
$$ 
Then $w_i\ge t_i$. Applying Theorem \ref{t-acc-glct}, we
can assume that the $w_i$ form a decreasing sequence. Then
$$
w_i \ge \lim w_i \ge \lim t_i = \mu_1 = \lim p_i \ge p_i
$$
 which implies that
$$
(X_i', B_i'+s_id_iS_i'+N_i'+p_iM_{1,i}')
$$
is generalized lc
with boundary part $\Delta_i' := B_i'+s_id_iS_i'$ and
nef part $R_i:=N_{i}+p_iM_{1,i}$.
The set of the coefficients of all the $\Delta_i'$ union the set 
$\{\mu_{j,i}\mid j\ge 2\}\cup \{p_i\}$ satisfies DCC.
Therefore, by Proposition \ref{p-global-acc-glc}, we may assume that $p_i$ is a constant independent of $i$.

Now
$$
\mu_1 = \lim p_i = p_i \le \mu_{1,i} \le \lim \mu_{1,i} = \mu_1
$$
Thus $p_i = \mu_{1,i}$, hence $s_i d_i = 0$, $\Delta_i' = B_i'$, and $R_i=M_i$. 
In other words, $(X_i',B_i'+M_i')$ is not generalized klt. This contradicts Proposition \ref{p-global-acc-glc}.

So after replacing the sequence we may assume that $S_i$ is exceptional over $X_i'$ for every $i$.\\

\emph{Step 8}.
By Lemma \ref{l-extract-div-2}, there is an extremal contraction $g_i\colon X_i''\to X_i'$ extracting $S_i''$
with $X_i''$ being $\Q$-factorial. We can assume $X_i\bir X_i''$ is a morphism.
We can write
$$
K_{X_i''}+B_{i}''+s_i\tilde{D}_i''+N_{i}''+t_iM_{1,i}''=g_i^*(K_{X_i'}+B_i'+s_iD_i'+N_i'+t_iM_{1,i}')\equiv 0
$$
where $B_{i}''$ is the pushdown of $B_i$, $\tilde{D}_i''$ is the birational transform of $D_i'$,
$M_{1,i}''$ is the pushdown of $M_{1,i}$, and $N_{i}''$ is the pushdown of $N_i$.
Now $S_{i}''$ is a component of $\rddown{B_{i}''}$.
By Lemma \ref{l-LMMP}(1) we can run the $-S_{i}''$-LMMP which terminates on some Mori fibre space
$X_i'''\to T_i'''$. We may assume that $\dim T_i'''=0$ for every $i$, or $\dim T_i'''>0$ for every $i$.
Replacing $X_i$ we may assume $X_i\bir X_i'''$
is a log resolution of $(X_i''',B_i'''+s_i\tilde{D}_i''')$.

Since $(X_i', B_i' + M_i')$ is generalized lc and $K_{X_i'}+B_i'+M_i'\equiv 0$,
we deduce that $K_{X_i}+B_i+M_i$ is pseudo-effective.
Thus $K_{X_i'''}+B_i'''+M_i'''$ is pseudo-effective too.
Moreover, by construction
$$
 K_{X_i'''}+B_i'''+ s_i\tilde{D}_i'''+N_i'''+t_iM_{1,i}'''\equiv 0 .
$$
So there is the largest
number $q_i\in [t_i,\mu_{1,i}]$ such that
$$
K_{X_i'''}+B_i'''+N_i'''+q_iM_{1,i}'''\equiv 0/T_i''' .
$$
From $s=\lim s_i=0$ we get $\lim t_i=\lim \mu_{1,i}=\mu_1$ from which we derive
$\lim q_i=\mu_1$. So we may assume that the $q_i$ form an increasing
sequence approaching $\mu_{1}$.
Let $w_i$ be the generalized lc threshold of $M_{1,i}'''$ with respect to
the generalized lc polarized pair $(X_i''', B_i''' + N_i''')$.
Then $w_i \ge t_i$ as
$$
({X_i'''},B_i'''+ s_i\tilde{D}_i'''+N_i'''+t_iM_{1,i}''')
$$
is generalized lc.
Moreover, by Theorem \ref{t-acc-glct} we can assume the $w_i$ form a decreasing
sequence, hence
$$
q_i \le \mu_{1,i} \le \mu_1 = \lim t_i \le \lim w_i \le w_i .
$$
So the pair $(X_i''', B_i''' + N_i''' + q_iM_{1,i}''')$
is generalized lc. But the pair is not generalized klt because $S_i'''$ is a component of $\rddown{B_i'''}$.\\

\emph{Step 9}.
Assume that $\dim T_i'''=0$ for every $i$. Applying Proposition \ref{p-global-acc-glc},
we can assume that the set of the coefficients of all the $B_i'''$ union the set
$\{\mu_{j,i}|j\ge 2\}\cup \{q_i\}$ is finite. In particular, this means we can assume
$q_i=\mu_{1,i}=\mu_{1}$ for every $i$, and that 
$\mu_{j,i}=\mu_j$ for every $j,i$ where $\mu_j:=\lim_i \mu_{j,i}$.
On the other hand, assume that $\dim T_i'''>0$ for every $i$.
If $M_{1,i}'''\equiv 0/T_i'''$, then $q_i=\mu_{1,i}$.
But if $M_{1,i}'''\not\equiv 0/T_i'''$, then by restricting to the general fibres of
$X_i'''\to T_i'''$ and applying induction,
we deduce that $\{q_i\}$ is finite, hence $q_i=\mu_1$ for $i\gg 1$;
so we can assume $q_i=\mu_{1,i}=\mu_1$.
Moreover, by restricting to the general fibres of $X_i'''\to T_i'''$ and applying induction once more,
we may assume that the set of the horizontal$/T_i'''$ coefficients of all the $B_i'''$ together
with the set $\{\mu_{j,i}\mid M_{j,i}'''\not\equiv 0/T_i'''\}$ is finite.

The last paragraph shows that in either case $\dim T_i'''=0$ or $\dim T_i'''>0$,
we can assume
$$
(**) ~~~~~~K_{X_i'''}+B_i'''+M_i'''\equiv 0/T_i''' .
$$
Let $\overline{B}_i$ be obtained from $B_i$ by replacing the coefficient
$b_{k,i}$ with $b_k:=\lim_i b_{k,i}$.
Let $\overline{M}_i$ be obtained from $M_i$ by replacing
$\mu_{j,i}$ with $\mu_j = \lim_i \mu_{j,i}$.
Then $K_{X_i'}+\overline{B}_i'+\overline{M}_i'$ is ample because $\rho(X_i') = 1$ and because
either $b_{k,i}<b_k$ for some $k$ or $\mu_{j,i}<\mu_j$ for some $j$.
Moreover, by Theorem \ref{t-acc-glct}, we can assume $({X_i'},\overline{B}_i'+\overline{M}_i')$
is generalized lc.
Thus
$$
K_{X_i}+\overline{B}_i+\overline{M}_i\ge f_i^*(K_{X_i'}+\overline{B}_i'+\overline{M}_i')
$$
is big. This in turn implies that
$K_{X_i'''}+\overline{B}_i'''+\overline{M}_i'''$ is big too. On the other hand,
by the last paragraph, we may assume that on the general fibres $F_i'''$ of $X_i'''\to T_i'''$ we have:
$\overline{B}_i'''|_{F_i'''}=B_i'''|_{F_i'''}$ and $\overline{M}_i'''|_{F_i'''}\equiv M_i'''|_{F_i'''}$.
This contradicts
$(**)$.\\
\end{proof}

%
%

\section{Proof of main results}

In this section, we prove our main results stated in the introduction.

\begin{proof}(of Theorem \ref{t-acc-glct} and Theorem \ref{t-global-acc})
By Proposition \ref{p-global-acc-to-glct}, Theorem \ref{t-global-acc} in dimension $<d$ 
implies Theorem \ref{t-acc-glct} in dimension d. On the other hand, by 
Proposition \ref{p-glct-to-global-acc}, Theorem \ref{t-global-acc} in dimension $<d$ 
 and Theorem \ref{t-acc-glct} in dimension $d$
imply Theorem \ref{t-global-acc} in dimension $d$. So both theorems follow inductively 
the case $d=1$ being trivial.\\
\end{proof}

Next we prove a result bounding pseudo-effective thresholds which will be needed 
for the proof of Theorem \ref{t-bir-bnd-M}.

\begin{thm}\label{t-peff-th-B+M}
Let $d$ be a natural number and $\Lambda$ a DCC set of nonnegative real numbers.
Then there is a real number $e\in (0,1)$
depending only on $\Lambda, d$ such that if:

$\bullet$ $(X,B)$ is projective lc of dimension $d$,

$\bullet$ $M=\sum \mu_jM_j$ where $M_j$ are nef Cartier divisors,

$\bullet$ the coefficients of $B$ and the $\mu_j$ are in $\Lambda$, and

$\bullet$ $K_X+B+M$ is a big divisor,\\\\
then $K_X + eB + eM$ is a big divisor.
\end{thm}
\begin{proof}
It suffices to show the assertion: there is an $e \in (0, 1)$ depending only on $\Lambda, d$ such that
$K_X + eB + eM$ is pseudo-effective; because then
$$\begin{aligned}
\vol(K_X + \frac{1}{2}(e+1)(B + M)) & = \vol(\frac{1}{2}(K_X + B + M + K_X + eB + eM)) \\
& \ge \vol(\frac{1}{2}(K_X + B + M)) > 0
\end{aligned}$$
and hence $K_X + e'B + e'M$ is big for $e':= \frac{1}{2}(1+e) \in (0, 1)$.

If there is no $e$ as in the last paragraph, then there is a sequence of pairs $(X_i,B_i)$ and divisors 
$M_i=\sum \mu_{j,i}M_{j,i}$
satisfying the assumptions of the theorem but such that the pseudo-effective thresholds $e_i$ of $B_i+M_i$
form a strictly increasing sequence approaching $1$: by definition $K_{X_i}+e_iB_i+e_iM_i$ is pseudo-effective
but $K_{X_i}+c_iB_i+c_iM_i$ is not pseudo-effective for any $c_i<e_i$.

We can extend $\Lambda$ and replace the $X_i,B_i$ so that we may assume $(X_i,B_i)$ is log smooth klt.
By Lemma \ref{l-LMMP}(2),
we can run an LMMP on $K_{X_i}+e_iB_i+e_iM_i$ which ends with a minimal model $X_i'$
on which $K_{X_i'}+e_iB_i'+e_iM_i'$ is semi-ample defining a contraction $X_i'\to T_i'$.
Since  $K_{X_i'}+B_i'+M_i'$ is big and $K_{X_i'}+e_iB_i'+e_iM_i'\equiv 0/T_i'$,
we deduce that $B_i'+M_i'$ is big over $T_i'$.

Replacing $X_i$ we may assume that $X_i\bir X_i'$ is a log resolution of $(X_i',B_i')$.
Let $F_i'$ be a general fibre of $X_i'\to T_i'$ and $F_i$ the corresponding fibre of
$X_i\to T_i'$. By restricting to $F_i'$ we get
$$
K_{F_i'}+e_iB_{F_i'}+e_iM_{F_i'}:=(K_{X_i'}+e_iB_{i}'+e_iM_{i}')|_{F_i'}\equiv 0 .
$$
This contradicts Theorem \ref{t-global-acc} because $e_iB_{F_i'}+e_iM_{F_i'}$ is big hence nonzero for every $i$,
so the set of the coefficients of all the $e_iB_{F_i'}$ union with the set
$\{e_i\mu_{j,i}\mid M_{j,i}|_{F_i}\not\equiv 0\}$ is not finite.\\
\end{proof}

\begin{proof}(of Theorem \ref{t-bir-bnd-M})
As usual by taking a log resolution we may assume $(X,B)$ is log smooth.
By Theorem \ref{t-peff-th-B+M},
there exist a rational number $e\in (0,1)$ depending only on $\Lambda, d,r$
such that $K_X+eB+eM$ is big, so $K_X+eB+M$ is also big. As in Step 2 of the proof of Proposition \ref{p-bir-nM-n-large},
there is $p\in \N$ depending only on $e,\Lambda, r$ such that $r|p$ and for any nonzero
$\lambda\in\Lambda$ we can find $\gamma\in [e\lambda, \lambda)$ such that $p\gamma$ is an integer.
In particular, we can find a boundary $\Delta$ such that $eB\le \Delta\le B$,
$p\Delta$ is Cartier, $K_X+\Delta+M$ is big, and $(X,\Delta)$ is klt. Replacing $B$ with $\Delta$ we can
then assume $\Lambda=\{\frac{i}{p} \mid 0\le i\le p-1\}$ and that $(X,B)$ is klt.

By Proposition \ref{p-bir-nM-n-large},
there exist $l,n\in \N$ depending only on $\Lambda, d,r$ such that $r|n$ and that
$|l(K_X+B+nM)|$ defines a birational map. By replacing $l$ with $pl$ we can assume $p|l$.
There is a resolution $\phi\colon W \to X$ such that
$$
\phi^*l(K_X+B+nM)\sim H+G
$$
where $H$ is big and base point free and $G\ge 0$. Perhaps after replacing $l$ with $(2d+1)l$, we
can also assume that $H$ is potentially birational [\ref{HMX}, Lemma 2.3.4].

Applying Theorem \ref{t-peff-th-B+M} once more, there exist rational numbers $s,u\in (0,1)$
depending only on $\Lambda, d,r$ such that $K_X+sB+uM$ is big.
Perhaps after replacing $s,u$,
we can choose a sufficiently large natural number $q$ 
so that $qs$ is integral and divisible by $p$, $qu$ is integral and divisible by $r$,
$$
s':=\frac{ qs+l}{q+l+1}<1,  ~~\mbox{and}~~ \,\,\,\, ~~~~ \frac{qu +ln}{q+l+1}=1 .
$$

Let $X'$ be a minimal model of $K_X+sB+uM$, which exists by Lemma \ref{l-LMMP}(2). We can assume that
the induced map $\psi\colon W\bir X'$ is a morphism. Since $X'$ is a minimal model,
$$
\phi^*(K_X+sB+uM)=\psi^*(K_{X'}+sB'+uM')+E
$$
where $E$ is effective.
 Let
$$
 D=\psi^*(K_{X'}+sB'+uM') .
$$
Since $H$ is potentially birational, by Lemma \ref{l-addNef},
$qD+H$ is potentially birational and $|K_W+\lceil qD+H \rceil|$ defines a birational map.
Thus
$$
|K_W+\lceil \phi^*q(K_X+sB+uM) \rceil +\phi^*l(K_X+B+nM)|
$$
also defines a birational map which in turn implies that
$$
|K_X+\lceil q(K_X+sB+uM) \rceil +l(K_X+B+nM)|
$$
defines a birational map. Hence the linear system
$$
|(q+l+1)(K_X+s'B+M)|
$$
defines a birational map. Therefore
$$
|(q+l+1)(K_X+B+M)|
$$
also defines a birational map. 

By construction $r|qu$ and $r|ln$, so $r|(q+l+1=qu+ln)$.
Now put $a:=m(\Lambda,d,r):=q+l+1$.
 Then $aM$ is Cartier, and for any 
$b\in \N$, the linear system $|b\rddown{a(K_X+B+M)}|$ defines a birational map. But since $aM$ is Cartier 
and $B$ is effective, 
$$
b\rddown{a(K_X+B+M)}\le \rddown{ba(K_X+B+M)}
$$
which means $|m(K_X+B+M)|$ also defines a birational map where $m=ba$.\\
\end{proof}

Next we prove a result similar to \ref{t-bir-bnd-M} but we allow a more general nef part $M$.
This result is not used elsewhere in this paper.

\begin{thm}\label{t-bir-bnd-M-general}
Let $d$ be a natural number and $\Lambda$ a DCC set of nonnegative real numbers.
Then there is a natural number $m$ depending only on $\Lambda, d$ such that if:

$\bullet$ $(X,B)$ is projective lc of dimension $d$,

$\bullet$ $M=\sum \mu_jM_j$ where $M_j$ are nef Cartier divisors,

$\bullet$ the coefficients of $B$ and the $\mu_j$ are in $\Lambda$, and

$\bullet$ $K_X+B+M$ is big,\\\\
then the linear system $|\rddown{m(K_X+B)}+\sum \rddown{m\mu_j}M_j|$ defines a birational map.
\end{thm}

\begin{proof}
As usual we may assume $(X,B)$ is log smooth.
By Theorem \ref{t-peff-th-B+M},
there exists a rational number $e\in (0,1)$ depending only on $\Lambda, d$
such that $K_X+eB+eM$ is pseudo-effective. As in the proof of \ref{t-bir-bnd-M},
there is $p\in \N$ depending only on $e,\Lambda$ such that
we can find a boundary $\Delta\le B$ and numbers $\nu_j\in [e\mu_j,\mu_j]$ such that
$p\Delta$ and $pN$ are Cartier divisors and $K_X+\Delta+N$ is big where $N=\sum \nu_jM_j$.

Applying Theorem \ref{t-bir-bnd-M}, there is $l\in \N$ depending only on $p,d$ (hence only
on $\Lambda,d$) such that $|l(K_X+\Delta+N)|$ defines a birational map and $p|l$. Replacing $l$
by a multiple we can in addition assume that $l(K_X+\Delta+N)$ is potentially birational.
Then by Lemma \ref{l-addNef},
$$
l(K_X+\Delta+N)+\sum \alpha_jM_j
$$
is potentially birational for any $0\le \alpha_j\in \Z$, and
$$
|K_X+l(K_X+\Delta+N)+\sum \alpha_jM_j|
$$
defines a birational map. Since $\nu_j\le \mu_j$, we can take $\alpha_j$
so that $l\nu_j+\alpha_j=\rddown{(l+1)\mu_j}$. Therefore
$$
|(l+1)K_X+{l}\Delta+\sum \rddown{(l+1)\mu_j}M_j|
$$
defines a birational map which in turn implies that
$$
|\rddown{(l+1)(K_X+B)}+\sum \rddown{(l+1)\mu_j}M_j|
$$
defines a birational map because $l\Delta\le \rddown{(l+1)B}$. Now put $m=l+1$.\\
\end{proof}

\begin{proof} (of Theorem \ref{ThA})
Replacing $W$ we can assume the Iitaka fibration $I\colon W\bir X$ is a morphism, i.e.
can assume $V=W$ using the notation before Theorem \ref{ThA}. Also we can assume $\kappa(W)\ge 1$
otherwise there is nothing to prove.

Let $b:=b_F$ and $\beta:=\beta_{\widetilde{F}}$. Let
$$
N = N(\beta) = \lcm \{m \in \N \, | \, \varphi(m) \le \beta\}
$$
where $\varphi$ denotes Euler's $\varphi$-function.
Let
$$
A(b, N) := \{\frac{bNu - v}{bNu} \, | \, u, v \in \N, \, v \le bN \}
$$
which is a DCC subset of the interval $[0, 1)$.

By the results of [\ref{FM}](which is summarized in [\ref{VZ}, Lemma 1.2]), replacing $W$ and $X$ by high enough resolutions,
we may assume that $X$ is smooth and that
there exist a boundary $B$ on $X$ (the discriminant part of $I : W \to X$)
and a nef  $\Q$-divisor $M$ (the moduli part of $I : W \to X$) such that

\begin{itemize}
\item $NbM$ is Cartier,
\item
$B$ has simple normal crossing support with coefficients in $A(b, N)$,
\item $K_X + B + M$ is big,
\item we have isomorphisms
$$
H^0(W, mb K_W ) \cong H^0(X, mb(K_X + B + M))
$$
for every $m \in \N$, and
\item
the rational map defined by ${|mbK_W|}$
is birational to the Iitaka fibration $I : W \to X$ if and only if $|mb(K_X + B +  M)|$
gives rise to a birational map.
\end{itemize}

By letting $\Lambda=A(b, N)$ and $r = Nb$, and applying Theorem \ref{t-bir-bnd-M},
there is a constant $m(\Lambda,d,r)$ depending only on $\Lambda, d,r,$ (hence depending only on 
$d,b, \beta$) such that $|m(K_X + B +  M)|$ defines
a birational map for any $m\in \N$ divisible by $m(\Lambda,d,r)$. Now simply let
$m(d,b_F, \beta_{\widetilde{F}})=bm(\Lambda,d,r)$.\\
\end{proof}

%



\vspace{2cm}

\textsc{DPMMS, Centre for Mathematical Sciences} \endgraf
\textsc{University of Cambridge,} \endgraf
\textsc{Wilberforce Road, Cambridge CB3 0WB, UK} \endgraf
\email{cb496@dpmms.cam.ac.uk\\}

\textsc{Department of Mathematics} \endgraf
\textsc{National University of Singapore,} \endgraf
\textsc{10 Lower Kent Ridge Road, Singapore 119076, Singapore} \endgraf
\email{matzdq@nus.edu.sg}

\end{document}